\documentclass[12pt,ltjsarticle]{amsart}
\usepackage[mathscr]{eucal}
\usepackage{amsmath}
\usepackage{amsthm}
\usepackage{amssymb}
\usepackage{amscd}
\usepackage{xfrac}
\usepackage{latexsym}
\usepackage{a4wide}
\usepackage{multirow}
\usepackage{color}
\theoremstyle{plain}
\newtheorem{themain}{Theorem}

\newtheorem*{theorem*}{Classical facts}
\newtheorem{theorem}{Theorem}[section]
\newtheorem{corollary}{Corollary}[section]
\newtheorem{lemma}{Lemma}[section]
\newtheorem{proposition}{Proposition}[section]
\theoremstyle{definition}
\newtheorem{remark}{Remark}[section]
\newtheorem{definition}{Definition}[section]

\newcommand{\diag}{\operatorname{diag}}
\newcommand{\h}{\mathbb H^3}

\newcommand{\R}{\mathbb R}
\newcommand{\C}{\mathbb C}

\newcommand{\D}{\mathbb D}
\newcommand{\Ht}{\mathbb H}
\newcommand{\Geo}{\mathrm{Geo}(\mathbb{H}^3)}
\newcommand{\p}{\partial }
\newcommand{\bp}{\bar \partial }
\newcommand{\dz}{\mathrm{d}z}
\newcommand{\dzb}{\mathrm{d}\bar z}
\newcommand{\SL}{\mathrm{SL}_{2}\mathbb{C}}
\newcommand{\slt}{\mathfrak {sl}_2 \mathbb C}
\newcommand{\ISU}{\mathrm {SU}_{1, 1}}
\newcommand{\SU}{\mathrm {SU}_2}
\newcommand{\isu}{\mathfrak {su}_{1, 1}}
\newcommand{\su}{\mathfrak {su}_2}
\newcommand{\LISU}{\Lambda \mathrm {SU}_{1, 1 \tau}}
\newcommand{\LSU}{\Lambda \mathrm {SU}_{2 \tau}}
\newcommand{\Uone}{\mathrm U_1}
\newcommand{\UH}{\mathrm{U}\mathbb{H}^3}
\newcommand{\Gr}{\mathrm{Gr}_{1,1}(\mathbb{E}^{1,3})}
\newcommand{\LSL}{\Lambda\mathrm{SL}_{2}\mathbb{C}_{\tau}}
\newcommand{\lslt}{\Lambda\mathfrak{sl}_{2}\mathbb{C}_{\tau}}

\numberwithin{equation}{section}
\def\Vec#1{\mbox{\boldmath $#1$}}

\newcommand{\eq}[1]{\begin{align*}#1 \end{align*}}

\renewcommand{\l}{\lambda}
\begin{document}
\title[CGC surfaces in $\mathbb{H}^3$]{A classification 
 of constant Gaussian curvature surfaces in the three-dimensional hyperbolic space}
\author[J.~Inoguchi]{Jun-ichi Inoguchi}
 \address{Department of Mathematics, 
Hokkaido University, Sapporo, 
060-0810, Japan}
 \email{inoguchi@math.sci.hokudai.ac.jp}
 \thanks{The first named author is partially supported by Kakenhi 24540063, 15K04834, 19K03461, 23K03081}
\author[S.-P.~Kobayashi]{Shimpei Kobayashi}
\address{Department of Mathematics, 
Hokkaido University, Sapporo, 
060-0810, Japan}
\email{shimpei@math.sci.hokudai.ac.jp}
\thanks{The second named author is partially supported by Kakenhi 
26400059, 18K03265, 22K03265}

\subjclass[2020]{Primary~53A10, 58D10, Secondary~53C42, 53C43}
\keywords{Constant Gaussian curvature; hyperbolic 3-space; loop groups}
\date{\today}
\begin{abstract}
 We classify weakly complete constant Gaussian 
 curvature   $-1<K<0$ surfaces in the hyperbolic three-space 
 in terms of holomorphic quadratic differentials.
 For this purpose, we first establish a loop group method for constant Gaussian curvature 
  surfaces with $K>-1$ and $K \neq 0$ via the harmonicity
 of the Lagrangian and Legendrian Gauss maps. 
 We then show that a spectral parameter deformation of the 
 Lagrangian harmonic Gauss map gives 
 a harmonic map into the hyperbolic two-space for $-1< K<0$ or the two-sphere for 
 $K>0$, respectively. Consequently, weakly complete constant Gaussian curvature surfaces with $-1 < K <0$ are in one-to-one correspondence with holomorphic quadratic differentials on the unit disk or the complex plane.
\end{abstract}
\maketitle

\section*{Introduction}
 As is well known, Hilbert 
 proved the nonexistence of isometric immersions of the hyperbolic plane $\mathbb H^2$
 into the Euclidean three-space $\mathbb{E}^3$. 
 Analogous results hold for surfaces in the unit three-sphere 
 $\mathbb{S}^3$ and the hyperbolic three-space $\mathbb{H}^3$ as follows; 
 see, for example, Spivak \cite{Spivak}:
\begin{theorem*}
\mbox{}
\begin{enumerate}
 \item 
 There is no complete surface in 
$\mathbb{S}^3$ with constant Gaussian curvature $K<1$ and $K\neq 0$, 
and the only complete immersions of constant Gaussian curvature $K\geq 1$
  in  $\mathbb{S}^3$ are totally umbilic round spheres.
 In particular, the only complete 
 surfaces with $K=1$ are great spheres.
 \item There is no complete surface in $\mathbb{H}^3$ with constant Gaussian curvature 
 $K <-1$, and the only complete immersions of constant Gaussian curvature $K\geq 0$ 
 in $\mathbb H^3$ are totally umbilic spheres or equidistant surfaces from geodesics
 $($only in the case $K=0)$. 
\end{enumerate}
\end{theorem*}
%

 On the one hand, complete surfaces with $K=0$ 
 in $\mathbb{S}^3$ \cite{Kit} and with $K=-1$ in $\mathbb{H}^3$ \cite{Nomizu} 
 are classical. 
 At the time when Spivak wrote his book \cite{Spivak}, the only known examples 
 of complete surfaces with constant Gaussian curvature $-1<K<0$ in $\mathbb{H}^3$ 
 were equidistant surfaces and surfaces of revolution. 
 Subsequently, Rosenberg and Spruck \cite{RS} proved that for any smooth curve 
in the ideal boundary of $\mathbb{H}^3$ and any $k\in(-1,0)$, there exists an embedded complete surface with Gaussian curvature $K=k$ whose boundary at infinity is 
the given curve. Other outstanding results about complete constant Gaussian curvature 
 surfaces solving a 
 Plateau problem at infinity were also given by Labourie \cite{La}.

 In our previous paper with Brander \cite{BIK}, we gave a loop group method 
 for constructing constant Gaussian curvature surfaces in $\mathbb{S}^3$ with 
  $K<1$ and $K\not=0$. 
 The key observation in \cite{BIK} is that every constant Gaussian curvature surface
 is uniquely determined by a Lorentzian harmonic 
 map into the unit two-sphere $\mathbb{S}^2$. The Lorentzian harmonic map is a
 normal Gauss map of the surface.
 On the other hand, in \cite{DIK1}, we established a unified loop group 
 method for constant mean curvature surfaces in $\mathbb{H}^3$. 
 The loop group method of \cite{DIK1} is based on the characterization of 
 constant mean curvature condition in terms of a Gauss map taking values in 
 the unit tangent sphere bundle $\UH$ of the hyperbolic three-space.

 The first purpose of the present paper is to establish a loop group method 
 for constant Gaussian curvature surfaces with $K>-1$ and $K \neq 0$ in $\mathbb{H}^3$.
 In contrast to the $\mathbb{S}^3$ case, there are no naturally associated 
the hyperbolic two-space-valued harmonic maps for constant Gaussian curvature surfaces in $\mathbb{H}^3$.
 Thus it is not easy to modify \cite{BIK} to the $\mathbb{H}^3$ case. 
 To overcome this difficulty, we first 
 establish a Ruh-Vilms-type characterization for 
 Legendrian and Lagrangian Gauss maps of constant Gaussian curvature surfaces in $\mathbb{H}^3$: 
\begin{themain}\label{thm:Ruh-Vilms}
Let $f: M \to \h$ be a surface with Gaussian curvature $K>-1$ 
and unit normal $n$.  Let $F=(f,n): M \to \UH$ and $L = f \wedge n: M \to \Geo$ be the Legendrian and Lagrangian 
 Gauss maps of $f$, respectively. 
 Then the following statements are mutually equivalent{\rm:}
 \begin{enumerate}
\item The Gaussian curvature of $f$ is constant. 
\item The Legendrian Gauss map $F$ is harmonic.
\item The Lagrangian Gauss map $L$ is harmonic.
\item The Klotz differential is holomorphic. 
\end{enumerate}
\end{themain}
 Note that $\Geo$ is the space of all oriented geodesics of $\mathbb H^3$. The precise definitions of the Klotz differential, Legendrian and Lagrangian 
 Gauss maps are given in Section \ref{sc:pre}, and the proof of Theorem 
 \ref{thm:Ruh-Vilms} is given in Section \ref{sbsc:Gauss}.
 From Theorem \ref{thm:Ruh-Vilms} and the structure equations 
 of a constant Gaussian curvature  surface $f$,
 there is a family of  constant Gaussian curvature surfaces $f^{\lambda}$, 
 the so-called  \textit{associated family}, parameterized by 
 the \textit{spectral parameter} $\lambda \in \mathbb S^1 \subset \mathbb C$ 
 such that $f^{\l}|_{\l=1}=f$. 

 The second purpose of the present paper is to characterize 
 constant Gaussian curvature surface with $K>-1$ and $K\neq 0$
 in terms of other simple harmonic maps, using a loop group method.
 For this purpose, we consider a deformation of the spectral parameter off the unit circle
 to the spectral parameter family of Lagrangian Gauss maps.
 We then prove that every constant Gaussian curvature 
 surface with $K>-1$ and $K\neq 0$ is recovered from the Lax-type equation 
 of harmonic maps into $\mathbb{H}^2$ and $\mathbb S^2$, respectively$:$
\begin{themain}\label{thm:realHarmonic}
 Let $f : M \to \h$ be a constant Gaussian curvature surface with 
 $K>-1$ and $K\neq 0$, 
 and let $f^{\l}$ denote 
 the associated family and $n^{\l}$ the corresponding family of unit normals to $f^{\l}$.
 Moreover, let $L^{\l} = f^{\l} \wedge n^{\l}$ be  a family of Lagrangian harmonic 
 Gauss  maps.
 Then the following statements hold$:$
\begin{enumerate}
 \item Assume $-1<K<0$ and 
 take a constant $\l_0 \in \C \setminus\bar{\D}$ with  
 $\bar{\D}= \{ \lambda \in \C \mid |\lambda|\leq 1 \}$ such that 
$ |\l_0|= \exp \big(\operatorname{arcosh} \sqrt{-1/K}\big) $
 holds.
 Then the map $L^{\l}|_{\l=\l_0}$  is a  
 harmonic local diffeomorphism into  $\mathbb H^2$ up to  conjugation.

\item   Assume $K>0$ and take a constant $\l_0 \in \C \setminus\bar{\D}$ such that 
$ |\l_0| =\exp \big(\operatorname{arsinh} \sqrt{1/K}\big) $
 holds.
 Then the map $L^{\l}|_{\l=\l_0}$ 
 is a harmonic local diffeomorphism into $\mathbb S^2$ up to conjugation.
\end{enumerate}
 Moreover, if the surface $f^{\l}$ is not totally umbilic,
 the harmonic Lagrangian Gauss map $L^{\l}|_{\l= \l_0}$ is non-conformal. 
\end{themain}
 The proof is given in Section  \ref{sbsc:dLG}.

  The final and principal purpose of the present paper is a classification of
 constant Gaussian curvature surfaces with $-1< K<0$. 
 By Theorems \ref{thm:Ruh-Vilms} and \ref{thm:realHarmonic}, 
 we obtain the harmonic local diffeomorphism $L^{\l}|_{\l=\l_0}$ 
 into $\mathbb H^2$
 corresponding to a constant Gaussian curvature surface $f$. It is known  \cite{TW} that all 
 complete $\p$-energy harmonic maps  
 from the unit disc $\D$ or the complex plane $\mathbb C$ into $\mathbb H^2$ 
 bijectively correspond  to holomorphic quadratic differentials. 
 In our case the holomorphic quadratic 
 differentials are exactly the Klotz differentials in  Theorem \ref{thm:Ruh-Vilms},
 however, the induced metric of $f$ is not related to 
 the $\p$-energy of the harmonic map.  
 We thus introduce another natural metric for a surface in $\mathbb H^3$ 
 in terms of the unit tangent sphere bundle $\UH$:
 The so-called \textit{weak metric} is defined by
\[
 \mathrm{d} t^2 = \langle \mathrm{d} f, \mathrm{d} f\rangle +\frac1{1+K} \langle\mathrm{d} n, \mathrm{d} n \rangle.
\]
 We then show the equivalence between the 
 completeness of the weak metrics and that of the $\p$-energy metric;
 see Lemma \ref{lm:completeness}.  
 Our main result of the paper is the following theorem.
\begin{themain}\label{thm:classification}
 Let \(\mathcal{K}\) denote the set of equivalence classes $($under rigid motions of \(\mathbb H^3\)$)$ of weakly complete constant Gaussian curvature surfaces with \(-1<K<0\), where the conformal type is either \(\mathbb D:=\{z\in\mathbb{C}\mid |z|\le 1\}\) or the complex plane \(\mathbb{C}\). Let \(\mathcal{QD}\) denote the set of equivalence classes $($under M\"obius transformations of \(\mathbb D\) or \(\mathbb{C}\)$)$ of holomorphic quadratic differentials on \(\mathbb D\) or \(\mathbb{C}\), not constant on \(\mathbb{C}\). Then there is a bijection between \(\mathcal{K}\) and \(\mathcal{QD}\).
\end{themain}
 The proof is given in Section 
 \ref{sbsc:classification}.
 In Corollary \ref{coro:equivariance}, 
 we show  that complete constant Gaussian curvature surfaces with $-1<K<0$
 are obtained by bounded holomorphic quadratic differentials.

 The paper is organized as follows: In Section \ref{sc:pre}, we  recall
 fundamental results and notions about surfaces in $\mathbb H^3$.
 In Section \ref{sc:harmonicGauss}, we discuss harmonicity
 of the Lagrangian and Legendrian Gauss maps for surfaces with Gaussian curvature 
 $K>-1$ and establish a loop group method for constant Gaussian curvature 
 surfaces.  In Section \ref{sc:Complete}, we classify weakly complete constant Gaussian curvature 
$-1< K <0$ surfaces  in terms of holomorphic quadratic differentials.

 After a previous version of this manuscript was circulated, 
 the related paper \cite{Kobayashi2} appeared, 
 developing the complex landslide flow via integrable-systems techniques.
%

\section{Preliminaries}\label{sc:pre}
 In this section, we recall fundamental results and notions about 
 surfaces with Gaussian curvature $K>-1$ in $\mathbb H^3$. 
 We first introduce conformal coordinates with respect to 
 the second fundamental form of such surfaces. Moreover, we recall 
 the unit tangent 
 sphere bundle $\UH$ of $\mathbb H^3$ and the space of all oriented geodesics $\Geo$, 
 and introduce the Lagrangian and Legendrian Gauss maps.

\subsection{Surfaces in the hyperbolic three-space}
 Let us denote by $\mathbb{E}^{1,3}$ the \textit{Minkowski four-space}, 
 that is,  $\mathbb{E}^{1,3}=(\mathbb{R}^4(x_0,x_1,x_2,x_3),
 -\mathrm{d}x_0^2+\mathrm{d}x_1^2+\mathrm{d}x_2^2+\mathrm{d}x_3^2)$.
 The \textit{hyperbolic three-space} $\mathbb H^3$ 
 is then defined by a unit central hyperquadric in $\mathbb{E}^{1,3}$:
\begin{align*}
\mathbb{H}^{3}=\{(x_0,x_1,x_2,x_3)\in \mathbb{E}^{1,3}
\ \vert \ -x_0^2+x_1^2+x_2^2+x_3^2=-1,\ \ x_0>0\}.
\end{align*}
 It is known that $\mathbb{H}^3$ is a simply connected Riemannian three-manifold 
 of constant curvature $-1$.
 Denote by $\mathrm{SO}^{+}_{1,3}$ the identity component of the Lorentz group 
 $\mathrm{O}_{1,3}$. Then $\mathbb{H}^3$ is represented by 
 $\mathbb{H}^3=\mathrm{SO}^{+}_{1,3}/\mathrm{SO}_3$.

 Let $f:M\to \mathbb{H}^3$ be an immersion from a surface $M$ into $\mathbb{H}^3$
 with unit normal $n \subset \mathbb{E}^{1,3}$ and 
 let us take a local orthonormal frame field 
 $\{e_1, e_2,e_3=n\}$ of $f$.  Moreover take a dual coframe field 
 $\{\omega^1, \omega^2, \omega^3\}$ of  $\{ e_1,e_2,e_3=n\}$. 
 This coframe field satisfies
\eq{
\mathrm{d} \omega^{\alpha}
=\omega^{\beta}\wedge \omega_{\beta}^{\ \alpha},
}
 where $1\leq \alpha,\beta\leq 3$, see, e.g., \cite{Tenenblat}.
 Note that we have used the Einstein summation convention.
 Moreover, the structure equation is 
\begin{equation}\label{eq:structure}
\mathrm{d} \omega_{\alpha}^{\ \beta}=
\> \omega_{\alpha}^{\ \gamma}\wedge 
\omega_{\gamma}^{\ \beta}+\Omega_{\alpha}^{\ \beta},
\end{equation}
 where $1 \leq \alpha, \beta, \gamma \leq 3$ and 
 $\Omega_{\alpha}^{\ \beta}$ is the curvature $2$-form of $\mathbb{H}^{3}$ 
 defined by 
\eq{
 \Omega_{\alpha}^{\ \beta}= \omega^{\alpha}\wedge \omega^{\beta}.
 }
 On the surface $(M,f)$, we have 
 $\omega^{i}\wedge \omega_{i}^{\ 3}=0, \;(1 \leq i \leq  2)$.
 Cartan's lemma implies that $\omega_{i}^{\ 3}$ is expressed as
 $\omega_{i}^{\ 3}=b_{ij}\>\omega^{j},\ \  b_{ij}=b_{ji}, \;(1 \leq i, j \leq 2)$.
 One can check that the first and second fundamental forms
 $\mathrm{I}=\langle \mathrm{d}f ,\mathrm{d}f\rangle$ 
 and $\mathrm{I\!I}=-\langle \mathrm{d}f ,\mathrm{d}n\rangle$ 
 are respectively represented as 
\begin{align}\label{eq:firstsecond}
\mathrm{I}=\omega^{1}\otimes \omega^{1} + \omega^{2}\otimes \omega^{2}, \;\;
\mathrm{I\!I}=b_{ij}\>\omega^{i}\otimes \omega^{j},
\end{align}
 where $1\leq i, j \leq 2$.
 Then, by \eqref{eq:structure} and \eqref{eq:firstsecond},
 the Gaussian curvature $K$ and the mean curvature $H$ can be computed as
\begin{equation}\label{eq:Gauss}
K=-1+\det(b_{ij}), \quad H = \frac12(b_{11} + b_{22}).
\end{equation}
 The third fundamental form $\mathrm{I\!I\!I} = \langle \mathrm{d} n, 
 \mathrm{d} n \rangle$ can be computed as
\begin{equation}\label{eq:third}
 \mathrm{I\!I\!I} = \omega_1^3 \otimes \omega_1^3 + 
\omega_2^3 \otimes \omega_2^3,
\end{equation}
 and \eqref{eq:firstsecond}, \eqref{eq:Gauss} and \eqref{eq:third}
 imply a relation
\begin{equation}\label{eq:123relation}
\mathrm{I\!I\!I} = 2 H \mathrm{I\!I}  -(K+1) \mathrm{I}.
\end{equation}
 From now on we assume that the Gaussian curvature $K$ satisfies $K>-1$. 
 Then, from \eqref{eq:firstsecond} and \eqref{eq:Gauss}, it is easy to see that the second 
 fundamental form $\mathrm{I\!I}$ derived from $n$ is a Riemannian metric on $M$.
 We equip a conformal structure $[\mathrm{I\!I}]$ determined by $\mathrm{I\!I}$. 
 This conformal structure is called the \textit{second conformal structure} of $M$ with 
 respect to the orientation determined by $n$. 
\begin{remark}
 If the Gaussian curvature $K$ satisfies $K<-1$, then the second fundamental form 
 $\mathrm{I\!I}$ defines a Lorentzian metric on $M$ and hence induce a Lorentz conformal 
 structure $[\mathrm{I\!I}]$ on $M$.
\end{remark}
 Using a positive function $\sigma$,  
 let us represent the Gaussian curvature $K>-1$ by 
 \eq{
 K=-1+\sigma^2
 } 
 and take a 
 conformal coordinate $z=x+yi$ with respect to the second conformal structure 
 $[\mathrm{I\!I}]$. 
 One can then check that the quadratic differential 
 \eq{Q \>\dz^2 = \langle \p f,\p f\rangle\>\dz^2} 
 is well defined on the Riemann surface $(M, [\mathrm{I\!I}])$, see,  e.g., 
 \cite{Klotz}.  Here we use the standard notation
 for operators $\p$ and $\bp$:
\begin{equation*}
 \p = \frac{1}{2} \left(\frac{\partial }{\partial x} - i \frac{\partial }{\partial y}\right), \quad \bp = \frac{1}{2} \left(\frac{\partial }{\partial x} + i \frac{\partial }{\partial y}\right).
\end{equation*}
 The quadratic differential $Q \> \dz^2$ is called
 the \textit{Klotz differential} of $(M,f)$, and we prove later that 
 the constancy of the Gaussian curvature $K>-1$
 is characterized by 
 holomorphicity of the Klotz 
 differential, see Theorem \ref{thm:Ruh-Vilms}. 
 
 From the above argument, the first and second fundamental forms are given by
 \eq{\mathrm{I}=Q \>\dz^2 +2 \ell \> \dz \dzb+ \bar Q \>\dzb^2, \;\;\;\;
 \mathrm{I\!I}= 2 m \> \dz \dzb,}
 where $\ell >0$ and $m>0$ are some positive functions.
 From these expressions, it is easy to
 see that $p \in M$ is an umbilic point if and only if $Q (p) =0$.
 Then, from \eqref{eq:Gauss} and the Gaussian curvature 
 $K = -1 +\sigma^2$, the equation 
 $\sigma^{-2} m^2  =  \ell^2- |Q|^2 $ 
 holds. Let us introduce a real function $u$ by $\ell + \sigma^{-1} m = e^{u}$.
 Then, from the equation 
 $\ell -\sigma^{-1} m = |Q|^2 e^{-u}$, we have 
\begin{equation*}
 \ell = \frac{1}{2}(e^u + |Q|^2 e^{-u}), \quad
 m = \frac{\sigma}{2  } (e^u - |Q|^2e^{-u}).
\end{equation*}
 Note that since $m >0$, we have 
 a condition $e^{2u}> |Q|^2$. Therefore the first and second 
 fundamental forms are rephrased as
\begin{equation}\label{eq:firstsecond2}
\mathrm{I}=Q \> \dz^2+(e^u +|Q|^2e^{-u})\> \dz \dzb
+\bar Q\>\dzb^2, \;\;\;
\mathrm{I\!I}=\sigma (e^u -|Q|^2e^{-u})\>\dz \dzb.
\end{equation}
 Note that determinant of the coefficient matrix of $\mathrm{I}$ 
 and the mean curvature $H$ can be computed respectively 
 as $(e^u - |Q|^2 e^{-u})^2$ and 
\begin{equation}\label{eq:H}
  H = \frac{\sigma(e^{2u} +|Q|^2)}{2(e^{2u} -|Q|^2)}.
\end{equation}
 Using the relation \eqref{eq:123relation}, we compute 
 the third fundamental form $\mathrm{I\!I\!I} =
 \langle \mathrm{d} n,  \mathrm{d} n\rangle$ as
\begin{equation}\label{eq:third2}
\mathrm{I\!I\!I} =\sigma^2(-Q \> \dz^2+(e^u +|Q|^2e^{-u})\> \dz \dzb
-\bar Q\>\dzb^2).
\end{equation}

\subsection{The $2\times{2}$-matrix model of $\mathbb H^3$}
Let 
$\{\Vec{e}_0,\Vec{e}_1,\Vec{e}_2,\Vec{e}_3\}$
be the standard basis of $\mathbb{E}^{1,3}$.
We identify $\Vec{e}_0$, $\Vec{e}_1$, 
$\Vec{e}_2$ and $\Vec{e}_3$ with the following $2\times 2$ matrices:
\begin{equation}\label{eq:e0123}
\Vec{e}_{0}=\left(
\begin{array}{cc}
1 & 0\\
0 & 1
\end{array}
\right),
\ \
\Vec{e}_{1}=\left(
\begin{array}{cc}
1 & 0\\
0 & -1
\end{array}
\right),
\
\Vec{e}_{2}=\left(
\begin{array}{cc}
0 & -i\\
i & 0
\end{array}
\right)
\ \mbox{and} \
\Vec{e}_{3}=\left(
\begin{array}{cc}
0 & 1\\
1 & 0
\end{array}
\right).
\end{equation}
 The Minkowski four-space $\mathbb{E}^{1,3}$ is identified with the space
 of all complex Hermitian $2\times{2}$-matrices $\mathrm{Her}_{2}\mathbb{C}$, 
 that is, $\mathbb E^{1,3} \cong \mathrm{Her}_2\mathbb{C}
 =\left\{\xi_0 \Vec{e}_0+ \xi_1 \Vec{e}_1+ \xi_2 \Vec{e}_2+ \xi_3 \Vec{e}_3
\
\bigr \vert
\
\xi_0,
\xi_1,\xi_2,\xi_3
\in \mathbb{R}\
\right\}$.
 It is easy to see that  
 for $\xi\in \mathrm{Her}_2\mathbb{C}$, 
 $-\det
 \xi=-\xi_0^2+\xi_1^2+\xi_2^2+ \xi_3^2$. Thus the Lorentzian metric
 of $\mathrm{Her}_2\mathbb{C}$ is described as 
\eq{
\langle \xi, \eta\rangle=-\frac{1}{2} {\mathrm{tr}}(\xi \Vec{e}_2 {}^t\eta \Vec{e}_2).
}
 In particular, we have
 $\langle \xi, \xi\rangle = -\det \xi$ for $\xi \in \mathrm{Her}_2\mathbb{C}$.
 Note that $\{\Vec{e}_0, \Vec{e}_1, \Vec{e}_2, \Vec{e}_3\}$
 is an orthonormal basis of $\mathbb E^{1, 3}$.
 We then have the identification: 
 $\mathbb{H}^3=\{ \xi \in \mathrm{Her}_2\mathbb{C} \ \vert \ \det \xi=1,\
 \mathrm{tr}\: \xi>0\}$.
 Moreover, the special linear group $\mathrm{SL}_2\mathbb{C}$ acts isometrically and
 transitively on the hyperbolic three-space via the action:
$\mathrm{SL}_2\mathbb{C}\times \mathbb{H}^3\to 
\mathbb{H}^3$ by $(g,\xi)\longmapsto g\> \xi \>g^{*}$.
 Here the superscript $*$ denotes 
 the Hermitian conjugate, that is,
 $g^* = {}^t \bar g$ for a matrix $g$.
 The isotropy subgroup of this action at $\Vec{e}_{0}=(1,0,0,0) \cong
 \mathrm{id}$ is the special unitary group $\mathrm{SU}_2$. Hence
 $\mathbb{H}^3$ is represented by
 $\mathbb{H}^3=\mathrm{SL}_2\mathbb{C}/\mathrm{SU}_2$ as a 
 \textit{Riemannian symmetric space}. 
 The natural projection $\pi:\mathrm{SL}_2\mathbb{C}\to
 \mathbb{H}^3$ is  explicitly given by $\pi(g)=gg^{*},\ g\in \mathrm{SL}_2\mathbb{C}$. In other
 words, $\mathbb{H}^3$ is represented as
$$
\mathbb{H}^{3}=\{gg^{*}\
\vert
\
g \in \mathrm{SL}_2\mathbb{C}\}.
$$
\begin{remark}\label{rm:Killingmetric}
 In this context, the Lie group
 $\mathrm{SL}_2\mathbb{C}$ is regarded as a simple real Lie group 
 and as a double covering of the special Lorentz group
 $\mathrm{SO}^{+}_{1,3}$. 
 The real Lie algebra $\mathfrak{sl}_2\mathbb{C}$ of $\mathrm{SL}_2\mathbb{C}$ 
 is spanned by the basis 
 $\mathcal{E}=\{i\Vec{e}_{1},i\Vec{e}_{2},i\Vec{e}_{3}, \Vec{e}_{1},\Vec{e}_{2},
 \Vec{e}_{3}\}$.
 With respect to the scalar product \eq{
 \langle X,Y\rangle_{K}=\frac{1}{2} \>\mathrm{tr}\>(XY),  \
 \;\;\;\left(X,Y \in \mathfrak{sl}_{2}\mathbb{C}\right),}
 $\mathcal{E}$ has signature $(-,-,-,+,+,+)$ and induces a bi-invariant 
 semi-Riemannian metric on $\mathrm{SL}_2\mathbb{C}$.
 The Lie algebras $\mathfrak{u}_1$ and $\mathfrak{su}_2$ 
are expressed as 
$\mathfrak{u}_1=\mathbb{R}(i\Vec{e}_1)$ and 
 $\mathfrak{su}_2=
\mathbb{R}(i\Vec{e}_1)\oplus\mathbb{R}(i\Vec{e}_2)
\oplus\mathbb{R}(i\Vec{e}_3)$,
respectively. Thus $\mathfrak{su}_2$ is negative definite 
linear subspace of $\mathfrak{sl}_2\mathbb{C}$.

The tangent space
 $\mathfrak{m}=T_{\Vec{e}_0}\mathbb{H}^3$ is given by
$$
\mathfrak{m}=\mathfrak{sl}_2\mathbb{C}
\cap \mathrm{Her}_{2}\mathbb{C}=
\mathbb{R}\Vec{e}_{1}\oplus \mathbb{R}\Vec{e}_{2}
\oplus \mathbb{R}\Vec{e}_3.
$$
One can see that the restriction of the scalar product 
$\langle\cdot,\cdot\rangle_K$ to $\mathfrak{m}$ coincides with the Lorentz scalar 
product on $\mathfrak{m}\subset \mathrm{Her}_2\mathbb{C}$. The induced Riemannian metric 
on $\mathbb{H}^3=\mathrm{SL}_2\mathbb{C}/\mathrm{SU}_2$ from 
$\langle\cdot,\cdot\rangle_K$ is of constant curvature $-1$.
\end{remark}
\subsection{Unit tangent sphere bundle}\label{sc:tangentspherebundle}
 The unit tangent sphere bundle $\mathrm{U}\mathbb{H}^3$ of $\h$ 
 is represented as,
\eq{
\mathrm{U}\mathbb{H}^{3}=
\{ (\Vec{x},\Vec{v})\in \mathrm{Her}_2\mathbb{C}\times 
\mathrm{Her}_2\mathbb{C}
\
\vert
\
\det \Vec{x}= -\det \Vec{v}=1, \mathrm{tr}\>\Vec{x}>0, \> \langle \Vec{x},\Vec{v}\rangle =0\},
}
 see \cite{DIK1} in detail.
 The special linear group $\mathrm{SL}_2\mathbb{C}$ acts
 isometrically and transitively on $\mathrm{U}\mathbb{H}^{3}$ via the
 action $g\cdot (\Vec{x},\Vec{v})=(g \Vec{x} g^{*}, g \Vec{v} g^{*})$.
 The isotropy subgroup of $\mathrm{SL}_2\mathbb{C}$ at $(\Vec{e}_{0},\Vec{e}_{1})$ is 
 \eq{H=\left\{
 \diag (e^{i\theta}, e^{-i\theta}) \mid \theta \in \R \right\}=\mathrm{U}_1.
}
 Thus the unit tangent sphere bundle $\mathrm{U}\mathbb H^3$ 
 is represented by $\mathrm{SL}_2\mathbb{C}/\mathrm{U}_1$ as a 
 homogeneous manifold.
 The Lie algebra $\mathfrak{h}$ of $H$ is
 $\mathfrak{h}=\left\{
 \diag (i a_2, -i a_2) \mid  a_2 \in \R\right\}$.
 The tangent space $\mathfrak{p}:=T_{(\Vec{e}_0,\Vec{e}_1)}
 \mathrm{U}\mathbb{H}^3$ is given by
\begin{equation}\label{eq:p}
\mathfrak{p}=
\left\{
\left(
\begin{array}{cc}
a_{1} & b_{1}+ib_{2}\\
c_{1}+ic_{2} & -a_{1}
\end{array}
\right)
\
\biggr
\vert
\
a_{1},b_{1},b_{2},
c_{1},c_{2}
\in \mathbb{R}
\>
\right\}.
\end{equation}
 Note that
\eq{
\mathfrak{p}=
\{\mathbb{R}\Vec{e}_1\oplus
\mathbb{R}\Vec{e}_2
\oplus
\mathbb{R}\Vec{e}_3\}
\oplus
\{
\mathbb{R}(i\Vec{e}_2)\oplus
\mathbb{R}(i\Vec{e}_3)
\}
= \mathfrak{m}\oplus 
\{
\mathbb{R}(i\Vec{e}_2)\oplus
\mathbb{R}(i\Vec{e}_3)
\}.
}
 The scalar product $\langle\cdot,\cdot\rangle_K$ induces a 
 $\mathrm{SL}_2\mathbb{C}$-invariant semi-Riemannian metric 
 on $\mathrm{U}\mathbb{H}^3$.  The resulting 
homogeneous semi-Riemannian space $G/H=\mathrm{SL}_2\mathbb{C}
/\Uone$ is naturally reductive. 
 The unit tangent sphere bundle admits two kinds of $\mathrm{CR}$-structures 
 $J_1$ and $J_2$. The resulting $\mathrm{CR}$-manifold $(\mathrm{U}\mathbb{H}^3,J_1)$
 is the \textit{twistor $\mathrm{CR}$-manifold} of $\mathbb{H}^3$ in the sense of LeBrun 
 \cite{LeBrun}.

\subsection{Natural projections and two Gauss maps}
 We have several natural projections from $\UH$. 
 To define the projections, we introduce the notion of 
 the space $\Geo$ of all oriented 
 geodesics in $\mathbb{H}^3$ which is identified with the Grassmannian 
 manifold $\Gr$ of all oriented timelike planes 
 in $\mathbb{E}^{1,3}$.
 Then, $\mathrm{U}\mathbb{H}^3$ is a principal line bundle
 over $\Gr$. 
 The natural projection $\pi_0:\mathrm{U}\mathbb{H}^3\to 
 \Geo$ is given explicitly by
 $\pi_0(\Vec{x},\Vec{v})=\Vec{x}\wedge \Vec{v}$.   
 The space $\Gr$ is a homogeneous space of 
 $G=\mathrm{SL}_2\mathbb{C}$, and  the isotropy subgroup $K$ 
 at $\Vec{e}_0\wedge\Vec{e}_1$ is 
\[
K=
\mathrm{GL}_1\mathbb{C}
=\{\mathrm{diag}(a,a^{-1})\>|\>a\in\mathbb{C}^{\times}\,\}
\cong \mathrm{GL}_1\mathbb{C}.
\] 
The Lie algebra $\mathfrak{k}$ of $K$ is
 $\mathfrak{k}=\left\{
 \diag (i a_1, - a_1) \>|\> a_1 \in \C \right\}$, and thus 
 the tangent space of $\Gr$ at $\Vec{e}_0\wedge\Vec{e}_1$ is 
identified with
\begin{equation}\label{eq:q}
\mathfrak{q}=
\left\{
\left(
\begin{array}{cc}
0 & b_{1}+ib_{2}\\
c_{1}+ic_{2} & 0
\end{array}
\right)
\
\biggr
\vert
\
b_{1},b_{2},
c_{1},c_{2}
\in \mathbb{R}
\>
\right\}.
\end{equation}
One can see that 
$G/K=\mathrm{SL}_{2}\mathbb{C}/\mathrm{GL}_1\mathbb{C}$ is a 
semi-Riemannian symmetric space with respect to the metric 
induced by $\langle\cdot,\cdot\rangle$. For more detail, see \cite{DIK1}. 

 Next, for $(\Vec{x}, \Vec{v}) \in \UH$, 
 it is easy to see that $\Vec{v}$ takes values in the \textit{de Sitter three-space}
 \eq{
 \mathbb S^{1,2} =\{(x_0,x_1,x_2,x_3)\in \mathbb{E}^{1,3}
\mid -x_0^2+x_1^2+x_2^2+x_3^2=1\},}
 which is a simply connected Lorentzian 
 three-manifold of constant sectional curvature $1$. Thus there are two 
 natural projections \[
\mbox{$\pi_+ : \UH \to \mathbb H^3$ by
 $\pi_+ (\Vec{x}, \Vec{v}) = \Vec{x}$, and  $\pi_- : \UH \to \mathbb S^{1,2}$
by $\pi_- (\Vec{x}, \Vec{v}) = \Vec{v}$},
\]
 respectively. Finally, there is also a natural projection from 
\[
 \mbox{$\pi_* : \SL \to \UH$ by 
 $\pi_*(g) = (gg^*, g\Vec{e}_1g^*)$. }
\]
 Note that the projection $\pi : \SL \to \h=\SL/\SU$
  is  given by the composition $\pi = \pi_+ \circ\pi_*$.
%

 Using these projections, we recall two Gauss maps for a surface in $\Ht^3$, which 
 are important in this paper. Note that several natural Gauss maps were already 
 defined; see \cite[Appendix of the archive version]{DIK1}.
\begin{definition}
 Let $f:M\to \mathbb{H}^3$ be a surface with unit normal $n$, and let 
 $\mathrm{Gr}_{1,1}(\mathbb{E}^{1,3}) = \Geo$.
 The {\rm Legendrian Gauss map} $F$ and the {\rm Lagrangian Gauss map} $L$ of $f$ are 
 respectively defined by 
 \eq{
F=(f,n):M\to \mathrm{U}\mathbb{H}^3 \quad\mbox{and}\quad
L=f\wedge n:M\to \mathrm{Gr}_{1,1}(\mathbb{E}^{1,3}).
}
\end{definition}
\begin{remark}
 Let us denote by $\eta$ and $\Omega$ the canonical contact 
 form of $\mathrm{U}\mathbb{H}^3$ and 
 the canonical symplectic form of $\mathrm{Gr}_{1,1}(\mathbb{E}^{1,3})$. 
 One can then see that $F$ is Legendre, that is, $F^{*}\eta=0$ and 
 $L$ is Lagrangian, that is, $L^{*}\Omega=0$, respectively, see 
 \cite{DIK1}.  In \cite{DIK1}, using the conformal structure induced from the first 
 fundamental form, we characterized 
 constant mean curvature surfaces by harmonicity of the Legendrian Gauss map.
\end{remark}

\section{Harmonic Gauss maps}\label{sc:harmonicGauss}
 In this section, we characterize constant Gaussian curvature $K>-1$ surfaces
 in terms of harmonicity of Legendrian and Lagrangian Gauss maps, respectively. 
 These harmonic Gauss maps naturally induce a one-parameter family, 
 the so-called \textit{spectral parameter family},  
 of  constant Gaussian curvature $K>-1$ surfaces. 
 Finally a deformation of the spectral parameter induces harmonic maps into 
 $\mathbb S^2$ in case of $K>0$ or $\mathbb H^2$ in case of  $-1< K < 0$, 
 respectively.
\subsection{Lax representation}
 Let $f:M\to \mathbb{H}^3$ be a surface wtih Gaussian curvature $K=-1+\sigma^2$
 for some positive function $\sigma$. 
 Let $D \subset M$ 
 be a simply connected domain and 
 $z= x + y i$ on $D$ a conformal coordinate on $D$ with respect to ${[\rm I\!I]}$. 
 We define two matrices as
\begin{equation}\label{eq:Vece}
 \Vec{\hat e}_2 = - \frac{1}{2}\left(\Vec{e}_2 -i \Vec{e}_3\right) = 
 \begin{pmatrix}
 0 & i \\ 0 & 0 
 \end{pmatrix}, \quad 
 \Vec{\hat e}_3 = - \frac{1}{2}\left(\Vec{e}_2 +i \Vec{e}_3\right)
 =  \begin{pmatrix}
 0 & 0 \\ -i  & 0 
 \end{pmatrix}. 
\end{equation}
 Note that $\Vec{\hat e}^*_3 = \Vec{\hat e}_2$ and 
 $\langle \Vec{\hat e}_i , \Vec{\hat e}_i \rangle = 0 \; (i = 2, 3)$
 and $ \langle \Vec{\hat e}_2 , \Vec{\hat e}_3 \rangle = 1/2 $, and 
 thus $\{\Vec{\hat e}_2, \Vec{\hat e}_3\}$ forms null basis. 
 Let us consider a change of basis $\{\Vec{e}, \Vec{e}^*\}$ 
 defined by 
\eq{
(\Vec{e}, \Vec{e}^*) = (\Vec{\hat e}_2, \Vec{\hat e}_3) 
 \begin{pmatrix} e^{u/2} & \bar Q e^{-u/2}\\ Q e^{-u/2}	& e^{u/2} \end{pmatrix}.
}
 It is easy to see that $\langle \Vec{e}, \Vec{e}\rangle =  Q$, 
 $\langle \Vec{e}^*, \Vec{e}^*\rangle =  \bar Q$ and 
  $\langle \Vec{e}, \Vec{e}^*\rangle =  (e^u + |Q|^2 e^{-u})/2$. 
 Since $\{\Vec{e}_0, \Vec{e}_1, \Vec{e}, \Vec{e}^*\}$ 
 is a basis in $\mathbb E^{1, 3}$, there exists 
 a map $\Psi:D\subset  M\to \mathrm{SL}_2\mathbb{C}$ 
 satisfying the following equations:
\eq{
f=\Psi \Vec{e}_0\Psi^{*}, \;
n=\Psi \Vec{e}_1\Psi^{*}, \; 
\p f=\Psi \Vec{e}\Psi^{*}\;\;\mbox{and}\;\;
\bp f =\Psi \Vec{e}^* \Psi^{*}.
}
 It is clear that $\Psi$ is a moving frame along the surface $f$. 
The moving frame $\Psi$ is called the \textit{coordinate frame}.
 We obtain the following \textit{Lax representation} for $f$.
\begin{proposition}\label{prop:UV}
 The coordinate frame $\Psi$ satisfies the Lax equation:
\begin{equation}\label{eq:Lax}
 \partial \Psi = \Psi U, \quad
 \bar \partial \Psi = \Psi V, 
\end{equation}
 where $U$ and $V$ are given by 
\begin{align}
 \label{eq:U1}
U&=\Big(\frac{1}{4}\p u+ p\Big) \Vec{e}_1 + \frac{1 + \sigma}{2} 
 (e^{u/2} \Vec{\hat e}_2 + Qe^{- u/2} \Vec{\hat e}_3 ),  \\
\label{eq:V1}
V &=-\Big(\frac{1}{4}\bp u+\bar p\Big) \Vec{e}_1
+\frac{1 - \sigma}{2}(\bar Qe^{- u/2} \Vec{\hat e}_2 + e^{u/2} \Vec{\hat e}_3 ),
\end{align}
 and $\Vec{e}_1$, $ \Vec{ \hat e}_2$ and $\Vec{\hat e}_3$ are defined in \eqref{eq:e0123} and \eqref{eq:Vece},
 and the function $p$ is given by $p=\frac12(- \bp Q + e^{-u} Q \p \bar Q)(e^u - |Q|^2e^{-u})^{-1}$.
 The compatibility conditions, that is, the Gauss-Codazzi equations for a surface 
 $f$,  are given as follows$:$
\begin{align}\label{eq:Gausseq}
\bp \p u
+\frac{K}{2}(e^u -|Q|^2 e^{-u})+ 2(\p \bar p+ \bp p)=0,  \\
2 \bp ( Q \sigma) + (e^u +|Q|^2e^{-u})\p \sigma=0.
\label{eq:Codazzieq}
\end{align}
 Conversely, let $u$ be a real smooth function and $Q$ a complex valued smooth function and $K = -1 + \sigma^2$ with some positive function 
 $\sigma$ be solutions
 of the Gauss-Codazzi equations \eqref{eq:Gausseq}and \eqref{eq:Codazzieq}
 such that $e^{2u} - |Q|^2>0$ holds, 
 and let $\Psi$
 be a solution of the Lax equation \eqref{eq:Lax}. Define
 respective maps
\eq{
 f = \Psi \Vec{e}_0\Psi^* \quad \mbox{and} \quad n= \Psi \Vec{e}_1\Psi^*.
}
 Then, $f$ is a surface in $\h$ with unit normal $n$ and the Gaussian curvature
 of $f$ is  $K = -1 + \sigma^2> -1$. 
\end{proposition}
\begin{proof}
 Differentiating $f = \Psi \Vec{e}_0 \Psi^*$ by $z$, 
 we have $\p f=\Psi(U+V^{*})\Psi^{*}$. Comparing it with the 
 formula $\p f=\Psi \left(e^{u/2} \Vec{\hat e}_2 + Qe^{- u/2} \Vec{\hat e}_3 \right) \Psi^*$, we obtain
\begin{equation}\label{eq:fz}
v_{11}=-\overline{u_{11}},
\ \ 
v_{22}=-\overline{u_{22}},
\ \ 
u_{12}+\overline{v_{21}}=i e^{u/2},
\ \ 
u_{21}+\overline{v_{12}}=-i Q e^{-u/2},
\end{equation}
 where we set $U =(u_{ij})_{1 \leq i, j \leq 2}$ and $V = (v_{ij})_{1 \leq i, j \leq 2}$. 
 Note that since $\Psi$ takes values in $\SL$, 
 we have $u_{11}+u_{22}=0= v_{11}+v_{22}$.
 Since $\langle \p f, \p n\rangle=0$ and $\p n =\Psi ( U \Vec{e}_1 + \Vec{e}_1 V^*) \Psi^*$, we have
\eq{
 \mathrm{tr}\{(e^{u/2} \Vec{\hat e}_2 + Qe^{- u/2} \Vec{\hat e}_3 )\cdot\Vec{e}_2{}^t(U\Vec{e}_1+\Vec{e}_1V^*)\Vec{e}_2\}=0,} 
 thus 
 $e^{u/2}(u_{21}-\overline{v_{12}})+Q e^{-u/2}(u_{12}-\overline{v_{21}})=0$.
 Using \eqref{eq:fz}, this is rephrased as
\begin{equation}\label{eq:u21}
 u_{21}+ Qe^{-u}u_{12}=0.
\end{equation}
 Next, by the second fundamental form $\mathrm{I\!I}$ in \eqref{eq:firstsecond} 
 we have $-\langle \bp f, \p n\rangle = \sigma(e^u -|Q|^2e^{-u})/2$.
 A direct computation similar to the above shows that this is equivalent with 
 $\sigma(e^u -|Q|^2e^{-u})/2 =
 i (\bar Q e^{-u/2}u_{21} + e^{u/2}u_{12}) - (e^u -|Q|^2 e^{-u})/2$.
 Comparing this with \eqref{eq:u21} and \eqref{eq:fz}, we obtain
\eq{
&u_{12}= \frac{i}{2}(1+ \sigma)e^{u/2}, 
\;\;\;u_{21}= -\frac{i}{2}(1+ \sigma)Q e^{-u/2},  \\
&v_{12}= \frac{i}{2}(1-\sigma)\bar Q e^{-u/2},\;\;\;
v_{21}= -\frac{i}{2}(1-\sigma)e^{u/2}.
}
 To obtain the diagonal part of $U$ and $V$, we use the compatibility condition
 $\bp \p f = \p \bp f$. A direct computation together with the relation 
 in \eqref{eq:fz} shows that this is equivalent to
\begin{align*}
-2 \overline{u_{11}} + \frac{\bp u}2 = 2 \bar Q e^{-u} u_{11} 
 + e^{-u} \Big(\bar Q_z - \frac{\p u}2 \bar Q\Big), \\
 -2 Q \overline{u_{11}} - \bp Q + \frac{\bp u}2 Q = \Big(2 u_{11}-\frac{\p u}2\Big)e^u. 
\end{align*} 
 From these equations, we obtain
\begin{equation}\label{eq:up}
 u_{11}= - \overline{{v}_{11}} = -\frac{\p u}{4}+p,
\ \ 
p=\frac{- \bp Q + e^{-u} Q \p \bar Q}
{2(e^u - |Q|^2e^{-u})}.
\end{equation}
 Now the compatibility conditions, that is $\bp\p\Psi = \p\bp\Psi$,
 are given by a straightforward computation: First note that 
 $\bp \p\Psi = \p\bp\Psi$ is equivalent to the zero curvature equation 
 $\mathcal{L}: = \bp U -\p V + [V, U] =0$.
 Then, $(1, 1)$- and
 $(2, 2)$-entries of $\mathcal{L}$ give the equation \eqref{eq:Gausseq}, that is, 
 the Gauss equation. Moreover, $(1, 2)$- and $(2, 1)$-entries of $\mathcal{L}=0$ can be computed as
\begin{align}
 i (\bp X- 2 \bar p X)e^{u/2} -i ( \bar Q \p Y+ Y \p \bar Q + 2 Y \bar Q p)e^{-u/2}=0,
\label{eq:1-2ent} \\
 i (- Q \bp X- X \bp Q- 2 X Q \bar p)e^{-u/2} -i (-\p Y+ 2 Y p)e^{u/2}=0,
\label{eq:2-1ent}
\end{align}
 where $X = (1+ \sigma)/2$ and 
 $Y = (1- \sigma)/2$.
 From these equations and the expression $p$ in \eqref{eq:up}, we have 
 the equation \eqref{eq:Codazzieq}, that is, the Codazzi equation.

 The converse statement is almost clear by a direct computation for 
 $f = \Psi \Vec{e}_0 \Psi^*$ with $n= \Psi \Vec{e}_1\Psi^*$. 
 The first fundamental form of $f$ is non-degenerate by the assumption, 
 that is, $f$ gives a surface. Then, $n$ is clearly the unit normal of $f$.
\end{proof}

\subsection{Harmonicity of Gauss maps}\label{sbsc:Gauss}
 We now characterize
 a constant Gaussian curvature surface with $K>-1$ in $\h$ 
 in terms of various Gauss maps as stated in Theorem \ref{thm:Ruh-Vilms}.
\begin{proof}[Proof of {\rm Theorem \ref{thm:Ruh-Vilms}}]
$(1) \Rightarrow (4)$: This is clear from the Codazzi equation, that is, 
 the equation \eqref{eq:Codazzieq}.
$(4) \Rightarrow (1)$: Conversely assume that $Q\> \dz^2$ is holomorphic. 
 Thus note that zeros of $Q$ are isolated. On a region such that $Q\neq 0$, 
 we consider the Codazzi equation \eqref{eq:Codazzieq}
 and its conjugation with multiplying $-(e^u+|Q|^2 e^{-u})/2$. 
 Adding it to the Codazzi equation multiplying $\bar Q$, 
 we obtain $-(\bp \sigma) (e^u-|Q|^2e^{-u})^2/2 =0$. 
 Since $\langle \bp f, \p n \rangle = \sigma (e^u-|Q|^2e^{-u})/2>0$, 
 this equation holds if and only if $\sigma$ is constant, thus the surface 
 is a constant Gaussian curvature surface with $K>-1$.
 We then use the argument of continuity and it implies that the surface is 
 a constant Gaussian curvature surface $K>-1$ on $M$.
%
%
%
%

 We now show that $(1) \Leftrightarrow (2)$.
 We use the harmonicity criterion \eqref{eq:harmonicity}
 for $F:M\to G/H=\mathrm{SL}_2\mathbb{C}/\mathrm{U}_1$ 
 with reductive decomposition 
 $\mathfrak{sl}_2\mathbb{C}=\mathfrak{u}_1\oplus\mathfrak{p}$, where 
 $\mathfrak{p}$ is given by \eqref{eq:p}.
  From Proposition \ref{prop:UV}, it is easy to see that $\alpha=\Psi^{-1}\mathrm{d}\Psi=
\alpha_{\mathfrak h} +\alpha_{\mathfrak p}^{\prime}+\alpha_{\mathfrak p}^{\prime\prime}$
 is given by 
\eq{\alpha_{\mathfrak h} = \Big(\frac{\p u}{4} + p\Big)\Vec{e}_1 \dz  
- \Big(\frac{\bp u}{4} + \bar p\Big)\Vec{e}_1 \dzb, \;\;
 \alpha_{\mathfrak p}^{\prime} =
 X \Vec{e} \; \dz 
\;\;\mbox{and } \;\;
\alpha_{\mathfrak p}^{\prime \prime}  =
 Y \Vec{e}^* \; \dzb,
}
%
%
%
 where $\Vec{e}$ is defined in \eqref{eq:Vece} and
 $X = (1+\sigma)/2$ and $Y =  (1-\sigma)/2$. 
 By a direct computation, we have 
$[\alpha^{\prime}_{\mathfrak p}\wedge 
\alpha^{\prime\prime}_{\mathfrak p}]=
 XY(e^u -|Q|^2e^{-u})\>\Vec{e}_1\>
\dz \wedge \dzb$.
From this it is clear that $[\alpha^{\prime}_{\mathfrak p}\wedge 
\alpha^{\prime\prime}_{\mathfrak p}]_{\mathfrak p}=0$.
 Next, we compute $ \mathrm{d}(*\alpha_\mathfrak{p})
+[\alpha\wedge *\alpha_{\mathfrak p}]$.
 It is easy to see that the $(1, 1)$- and $(2, 2)$-entries are zero.
 Then, the $(1, 2)$- and $(2, 1)$-entries can be computed as
\eq{&
\left\{ (2 \bar p  X - \bp X)e^{u/2}-(2 p \bar Q Y + \p (Y\bar Q)) e^{-u/2}\right\}
 \dz \wedge \dzb, \\
&\left\{ (\bp (X Q) + 2 \bar p Q X)e^{-u/2}+(\p Y-  2 p Y) e^{u/2}\right\} \dz \wedge \dzb.
}
 Using the equations \eqref{eq:1-2ent} and \eqref{eq:2-1ent}, $ \mathrm{d}(*\alpha_\mathfrak{p})
+[\alpha\wedge *\alpha_{\mathfrak p}] = 0$ is equivalent with 
\eq{
\bp X - 2 \bar p X = \p Y + 2p Y =0.
}
 These equations hold if and only if $p=0$ and 
 this happens if and only if $\bp Q =0$, that is, the Klotz differential 
 is holomorphic. Since $(1)$ is equivalent with $(4)$, the claim follows.
 
 Finally we show $(1) \Leftrightarrow (3)$. 
 Let us consider the harmonicity of $L:M\to 
 G/K=\mathrm{SL}_2\mathbb{C}/\mathrm{GL}_1\mathbb{C}$ 
 equipped with the reductive decomposition 
 $\mathfrak{sl}_2\mathbb{C}=\mathfrak{k}\oplus\mathfrak{q}$, where 
 $\mathfrak{q}$ is given by \eqref{eq:q}. Since $\Geo$ is a symmetric space, 
 the first condition in \eqref{eq:harmonicity}, that is, 
 $[\alpha^{\prime}_{\mathfrak q}\wedge \alpha^{\prime\prime}_
 {\mathfrak q}]_{\mathfrak q}=0$ is vacuous. 
 The harmonicity of $L$ is exactly the same as the case $(2)$ and the claim follows.
 This completes the proof.
\end{proof} 
\begin{corollary}\label{coro:flatconnections}
 Let $U^{\l}$ and $V^{\l}$ be a $\mathbb S^1$-family of matrices 
 parameterized by $\lambda \in \mathbb S^1$$:$ 
\begin{align}\label{eq:UVlambda1} 
U^{\l}&=\Big(\frac{1}{4}\p u+ p\Big) \Vec{e}_1 + \frac{1 + \sigma}{2} 
 \l^{-1}(e^{u/2} \Vec{\hat e}_2 + Qe^{- u/2} \Vec{\hat e}_3),  \\
\label{eq:UVlambda2} 
V^{\l} &=-\Big(\frac{1}{4}\bp u+\bar p\Big) \Vec{e}_1
 + \frac{1 - \sigma}{2}\l (\bar Qe^{- u/2} \Vec{\hat e}_2 + e^{u/2} \Vec{\hat e}_3).
\end{align}
%
 Moreover, set a $\mathbb S^1$-family of Maurer-Cartan forms as
\begin{equation*}
 \alpha^{\l} = U^{\l} \mathrm{d}z+V^{\l} \mathrm{d}\bar{z}.
\end{equation*}
 Any statement in {\rm Theorem \ref{thm:Ruh-Vilms}} is then equivalent with the following
 statement$:$
 \begin{enumerate}
 \item[(5)] $\mathrm{d} + \alpha^{\l}$ is a family of flat connections on $M \times \SL$.
\end{enumerate}
\end{corollary}
\begin{proof}
 The family of Maurer-Cartan form $\alpha^{\lambda}$ can be decomposed as 
\begin{equation*}
 \alpha^{\lambda} = 
\lambda^{-1} \alpha_{\mathfrak p}^{\prime}
+ \alpha_{\mathfrak k} + 
\lambda \alpha_{\mathfrak p}^{\prime \prime},
\end{equation*}
 according to the decomposition of $\mathfrak g = \mathfrak k + \mathfrak p$
 of the symmetric space $\Geo = G/K$ and the complex structure of $M$. 
 Here note that $\lambda^{-1}\alpha_{\mathfrak p}^{\prime} = \operatorname{Off} (U^{\lambda}) \dz$
 and $\lambda \alpha_{\mathfrak p}^{\prime \prime} =  \operatorname{Off} (V^{\lambda})\dzb$, where  $\operatorname{Off}(X)$ denotes the off-diagonal part of $X$.
 It is well known that, see for example Theorem \ref{thm:flatconnections}
 in Appendix \ref{ap:harmonic}, $\mathrm{d} + \alpha^{\l}$ is flat if and only if 
 the corresponding Lagrangian Gauss map $L: M \to \Geo$ is harmonic.
\end{proof}
\begin{remark}
\mbox{}
\begin{enumerate}
 \item 
 If the surface is a flat surface $K =0$, then the harmonic 
 Legendrian and Lagrangian Gauss maps become conformal.
 In fact a direct computation shows 
 that 
 $\langle \p F,  \p F \rangle  = \langle \p L,  \p L \rangle =0$.
 Moreover under the identification, $\mathrm{Geo}(\mathbb {H}^3)=\mathbb{S}^2\times
 \mathbb{S}^2\setminus\Delta$, the Lagrangian Gauss map 
 $L$ is the pair of hyperbolic Gauss maps.
 For flat surfaces, G{\'a}lvez et al \cite{GMM:flat} 
 showed that both hyperbolic Gauss maps are 
 conformal.

 \item Here $\SL$ is the double cover of $\mathrm{SO}^{+}_{1,3}$
 and we consider it as a real Lie group. Thus $\alpha_{\mathfrak p}^{\prime}$-
 and $\alpha_{\mathfrak p}^{\prime \prime}$-parts of the Maurer-Cartan form 
 in the proof of Theorem \ref{thm:Ruh-Vilms} should be in 
 the complexification of $\slt$, that is, it is given by $\slt \times 
 \slt$, see, e.g., \cite[Section 7.1]{DIK1}. 

\item  The family of flat connections $\mathrm{d} + \alpha^{\l}$ induces 
 a family of maps $\{ \Psi^{\l}\}_{\l \in \mathbb S^1}$ such that
 $(\Psi^{\l})^{-1} \mathrm{d} \Psi^{\l} = \alpha^{\l}$, the so-called  
 extended frame, see Section \ref{subsc:spectral}. 
\end{enumerate}
\end{remark}

\subsection{Spectral parameter}\label{subsc:spectral}
 We now show that the parameter $\l \in \mathbb S^1$ 
 in Corollary \ref{coro:flatconnections} 
 naturally appears in the Gauss-Codazzi equations of 
 a constant Gaussian curvature $K>-1$ surface.
 From \eqref{eq:Gausseq} and \eqref{eq:Codazzieq}, 
 they are given as
\begin{equation}\label{eq:GCforCGC}
\bp \p u+\frac{K}{2}( e^u -|Q|^2e^{-u})=0\quad\mbox{and}\quad \bp Q=0.
\end{equation}
 In particular, note that 
 the function $p=(- \bp Q + e^{-u} Q \p \bar Q)(2e^u -2|Q|^2e^{-u})^{-1}$ vanishes by holomorphicity of $Q$.
\begin{remark}\label{rm:harmoniceq}
 Since the Klotz differential $Q\>\mathrm{d}z^2$ is holomorphic, 
 on a (simply connected) region free of umbilics, we may reparameterize $z$ 
 so that $Q=1$. With respect to such a coordinate $z$, the Gauss-Codazzi equation 
 has the form{\:}
\eq{
\bp \p u+K\sinh u=0, \ \ (K>-1). 
}
This equation appears as a harmonic map equation for 
 $\mathbb{D}\subset \mathbb{C}\to \mathbb{H}^2$ in case of $-1<K<0$ 
 or $\mathbb{S}^2$ in case of $K>0$, respectively, see \cite{Kobayashi}.
\end{remark}
 It is easy to see that the Gauss-Codazzi equations \eqref{eq:GCforCGC}
 are invariant under the 
 deformation $Q\longmapsto \lambda^{-2}Q$ for $\lambda\in \mathbb{S}^1$.
 Then, from fundamental theory of surface, there is a $\mathbb S^1$-family 
 of constant Gaussian curvature $K>-1$ surfaces in $\Ht^3$. 
 Let $\widetilde \Psi$ be the corresponding family of moving frames, 
 that is, 
\begin{equation}\label{eq:tildePsi}
 \widetilde \Psi^{-1} \mathrm{d}\widetilde \Psi = 
 U \dz + V \dzb,
\end{equation}
 where $U$ and $V$ and defined in \eqref{eq:U1} and \eqref{eq:V1}
 with $p = 0$,  $\sigma$ constant and the holomorphic $Q$ and anti-holomorphic 
 $\bar Q$ are replaced by $\l^{-2} Q$ and $\l^{2} \bar Q$,
 respectively. Consider a gauge transformation by right:
\begin{equation*}
 \widetilde \Psi \longmapsto  
 \Psi^{\l} = \widetilde \Psi \diag(\l^{1/2},  \l^{-1/2}).
\end{equation*}
 A straightforward computation shows that 
 $U^{\l}=(\Psi^{\l})^{-1} \p \Psi^{\l}$ and $V^{\l}=(\Psi^{\l})^{-1} \bp \Psi^{\l}$
 are given in \eqref{eq:UVlambda1} and \eqref{eq:UVlambda2} 
 with $p = \bar p = 0$,  respectively. In this way the parameter $\l \in \mathbb S^1$
 appears in the Gauss-Codazzi equations of a constant Gaussian curvature surface.
\begin{definition}
  The solution $\Psi^{\l}$ of 
 \begin{equation}
  (\Psi^{\lambda})^{-1}\mathrm{d}\Psi^{\l}=
 U^{\l} \mathrm{d}z+V^{\l} \mathrm{d}\bar{z}\;\;\;
 \mbox{with} \;\;\;\Psi^{\l}|_{z= z_*} =  {\rm id}
 \end{equation}
 is called the 
 \textit{extended frame} of a constant Gaussian curvature $K>-1$ surface $f$
 and the Lagrangian harmonic Gauss map $L : D \to \Geo$, 
  where $z_*$ is some base point $z_* \in D$. Moreover, the
  parameter $\l \in \mathbb S^1$ in the extended frame $\Psi^{\l}$
  is called the \textit{spectral parameter}. 
\end{definition}
  One can easily check that 
 $U(\l) = U^{\l}$ and $V(\l) = V^{\l}$
 satisfy the following conditions$:$
\eq{ 
 \mathrm{Ad}(\Vec{e}_1)U({\l})=U({-\l}), 
\quad 
 \mathrm{Ad}(\Vec{e}_1)V(\lambda)=V(-\lambda).
} 
 Thus the extended frame $\Psi^{\l}$ takes values in the twisted loop group 
 $\LSL$, see Appendix \ref{ap:loopgroups} for the definition.
 From the extended frame $\Psi^{\l}$, it is easy to obtain a 
 $\mathbb S^1$-family of constant Gaussian curvature $K>-1$ surfaces.
 \begin{proposition}\label{prp:associatefamily}
 Let $\Psi^{\l}$ be the extended frame of some constant Gaussian curvature $K>-1$ 
 surface $f$
 and define respective maps $f^{\l}$ and $n^{\l}$ by
\eq{
  f^{\l} =\Psi^{\l} \Vec{e}_0 (\Psi^{\l})^* \quad\mbox{and}\quad
  n^{\l} =\Psi^{\l} \Vec{e}_1 (\Psi^{\l})^*.
}
  Then, for each $\l \in \mathbb S^1$, $f^{\l}$ is a constant Gaussian curvature 
 $K>-1$ surface with unit normal $n^{\l}$ and 
 the following first and second fundamental forms$:$ 
\eq{
\mathrm{I^{\l}}=\l^{-2} Q \> \dz^2+(e^u +|Q|^2e^{-u})\> \dz \dzb
+\l^{2} \bar Q\>\dzb^2, \quad 
\mathrm{I\!I}^{\l}=\sigma(e^u -|Q|^2e^{-u})\>\dz \dzb.
}
 In particular $f^{\l}|_{\l=1}$ and the original 
 constant Gaussian curvature surface $f$
 are the same surface up to rigid motion.
 Moreover, the maps 
 \eq{
F^{\l}=(f^{\l},n^{\l}):M\to \mathrm{U}\mathbb{H}^3 \quad \mbox{and} \quad
L^{\l} =f^{\l}\wedge n^{\l}:M\to \mathrm{Gr}_{1,1}(\mathbb{E}^{1,3})
}
 are the Legendrian harmonic map and the Lagrangian harmonic map of 
 $f^{\l}$, respectively.
\end{proposition}
\begin{proof}
 Since $\l \in \mathbb S^1$, we have 
\eq{
 f^{\l} = \Psi^{\l} \Vec{e}_0 (\Psi^{\l})^* =
\widetilde \Psi \Vec{e}_0 \widetilde \Psi^*,
}
 where $\widetilde \Psi$ is defined in \eqref{eq:tildePsi}.
 Now the claim follows directly from Proposition \ref{prop:UV}.
\end{proof}
\begin{definition}
 A family of surfaces $\{f^{\l}\}_{\l \in \mathbb S^1}$ defined in Proposition
 \ref{prp:associatefamily} is called the {\it associated family} of the 
 constant Gaussian curvature $K>-1$ surface $f(=f^{\l}|_{\l=1})$.
\end{definition}

\subsection{A spectral parameter deformation of the harmonic Lagrangian Gauss maps}\label{sbsc:dLG}
 From Remark \ref{rm:harmoniceq},  
 for a constant Gaussian curvature surfaces with $K>-1$ but $K\neq 0$ in $\h$, 
 it is expected that there are harmonic maps 
 into $\mathbb{H}^2=\ISU/\Uone$ or $\mathbb S^2 
 = \SU/\Uone$ which are obtained from the constant Gaussian curvature surface. 
 Thus it is also expected that a family of frames taking values 
in the loop group of  $\ISU$ or $\SU$, that is, $\LISU$ or $\LSU$, 
 see Appendix \ref{ap:loopgroups} for the definitions.
 Unfortunately, the extended frame $\Psi^{\l}$ of a constant Gaussian curvature 
 $K>-1$ surface does \textit{not} take values in $\LISU$ nor $\LSU$
 for any $\l \in \mathbb S^1$.
 Thus we consider a deformation of the spectral parameter: 
\eq{
 \lambda \in \mathbb S^1 \to  \lambda \in \mathbb C^{\times} (:= \{ c \in \mathbb C \;|\;
 c \neq 0\}).
}
 The extended frame $\Psi^{\l}$ still takes values of $\LSL$.
 In Theorem \ref{thm:realHarmonic} stated in Introduction, we claim that there exist
 two special absolute values  $|\l_0| \neq 1$ of the spectral parameter
 such that the Lagrangian Gauss maps take values $\mathbb S^2$ or $\mathbb H^2$ 
 at $\lambda_0$. 
 We now prove Theorem \ref{thm:realHarmonic}:
%
%
%

\begin{proof}[Proof of {\rm Theorem \ref{thm:realHarmonic}}]
 We first rephrase the conditions of $\lambda_0$.

 $-1< K<0$ and by using a positive $\sigma>0$ such that $K = -1 + \sigma^2$ holds:
 For $-1< K< 0$, that is, $0< \sigma<1$, the 
 condition $|\l_0|=\exp \big(\operatorname{arcosh} \sqrt{-1/K}\big)$ 
 is equivalent to 
\begin{equation}\label{eq:lambdacond2}
|\l_0|= \sqrt{\frac{1+ \sigma}{1- \sigma}}.
\end{equation}
 For $K>0$, that is, $\sigma>1$, the condition 
 $|\l_0|=\exp \big(\operatorname{arsinh} \sqrt{1/K}\big)$ 
 is equivalent to 
\begin{equation}\label{eq:lambdacond}
|\l_0| = \sqrt{\frac{1+\sigma}{-1+ \sigma} }.
\end{equation}
 Let us represent the Lagrangian Gauss map 
 by 
\[
L^{\l} =  i \Psi^{\l}\Vec{e}_1 (\Psi^{\l})^{-1}.
\]
 It is easy to see that 
 $L^{\l}|_{\l \in \C^{\times}}$ still takes values in the symmetric space 
 $\SL/\mathrm{GL}_1\mathbb{C}$.
 Then, $L^{\l}|_{\l = \l_0 \in \C^{\times}}$ 
 takes values in $\mathbb S^2$ (resp. $\Ht^2$), up to conjugation by a
 constant factor in $\SL$ if and only if $\Psi^{\l}|_{\l = \l_0 \in \C^{\times}}$
 takes values in $\SU$ (resp. $\ISU$), up to conjugation by the
 constant factor in $\SL$. 
 Let $U(\l) = U^{\l}$ and $V(\l) = V^{\l}$ be the matrices 
 defined in \eqref{eq:UVlambda1} and \eqref{eq:UVlambda2}, respectively.
 It is clear that $\Psi^{\l}|_{\l = \l_0 \in \C^{\times}}$
 takes values in $\SU$ (resp. $\ISU$), up to conjugation by a
 constant factor in $\SL$, if and only if  
\eq{
 - {}^t \overline {U(\l_0)} = V(\l_0), \ \ \
 (\mbox{resp.}\  -\Vec{e}_1{}^t \overline{U(\l_0)}\Vec{e}_1 = V(\l_0))
}
 holds. A straightforward computation shows that this is equivalent with  
 the equation \eqref{eq:lambdacond2} (resp. \eqref{eq:lambdacond}).

 In case of $\mathbb S^2$, $\l_0$ in \eqref{eq:lambdacond2} 
 exists if and only the surface has constant Gaussian curvature 
 $K= -1 + \sigma^2>0$.
 In case of $\mathbb H^2$, $\l_0$ in \eqref{eq:lambdacond}  exists
 if and only if  the surface has constant Gaussian curvature $-1<K = -1 + \sigma^2< 0$.
 Moreover $\l_0$ is unique up to phase. Since harmonicity of the 
 Lagrangian Gauss map does not depend on the values of $\l \in \C^{\times}$, 
 that is, for any $\l \in \C^{\times}$, the Lagrangian Gauss map 
 $L^{\l}$ is a harmonic map, 
 thus $L^{\l}|_{\l = \l_0}$ is a harmonic map into $\mathbb S^2$
 or $\mathbb H^2$.

 Consider the inner product $\langle \mathrm{d} L^{\l}, \mathrm{d} L^{\l} \rangle$. Here 
 $\langle \>,\> \rangle$
 denotes the Killing metric on $\su$ or $\isu$:
\eq{
\langle A , B \rangle_{\su} = - \frac{1}{2} \operatorname{tr} AB \ \ \mbox{or}\ \
\langle A , B \rangle_{\isu} = \frac{1}{2} \operatorname{tr} AB, 
} see also Remark \ref{rm:Killingmetric}.
 Using the form of $U^{\l}$ and $V^{\l}$, we compute 
 \eq{
 \p L^{\l} &= 2i \l^{-1}X \Psi^{\l} (-e^{u/2} \Vec{\hat e}_2 + Q e^{-u/2} \Vec{\hat e}_3)(\Psi^{\l})^{-1}
 \\
 \bp L^{\l} &= 2 i \l Y   \Psi^{\l} (-\bar Q e^{-u/2}\Vec{\hat e}_2+ e^{u/2} \Vec{\hat e}_3) (\Psi^{\l})^{-1},
}
 where $X =  (1+ \sigma)/2, Y = (1-\sigma)/2$ and $\Vec{\hat e}_2, \Vec{\hat e}_3$ 
 are  defined in \eqref{eq:1-2ent}, \eqref{eq:2-1ent} and \eqref{eq:Vece}, respectively. 
 In case of $\mathbb H^2$,
 that is, in case of $-1< K<0$,  using the condition in \eqref{eq:lambdacond2}, 
 we compute 
 \eq{
\langle \p L^{\l}, \p L^{\l}\rangle_{\isu}\big|_{\l = \l_0} 
= - K e^{-2i \theta}Q, \quad 
 \langle \p L^{\l}, \bp L^{\l}\rangle_{\isu}\big|_{\l = \l_0} 
 = -K \frac{e^u + |Q|^2e^{-u}}{2},
} 
 and $\langle \bp L^{\l}, \bp L^{\l} \rangle_{\isu}\big|_{\l = \l_0}  
 =  - K e^{2i \theta}\bar Q$.
 Similarly, in case of $\mathbb S^2$,
 that is, in case of $K>0$, using the condition in \eqref{eq:lambdacond}, we compute 
 \eq{
\langle \p L^{\l}, \p L^{\l}\rangle_{\su}\big|_{\l = \l_0} = 
 - K e^{-2i \theta}Q, \quad
\langle \p L^{\l}, \bp L^{\l}\rangle_{\su}\big|_{\l = \l_0} = K \frac{e^u + |Q|^2e^{-u}}{2},
}
and $\langle \bp L^{\l}, \bp L^{\l}\rangle_{\su} \big|_{\l = \l_0} =- K e^{2i \theta}\bar Q $.
 From these equations it is easy to see that 
 $L^{\l}|_{\l = \l_0}$ is a local diffeomorphism.
 Moreover if $f$ is a non-totally umbilic constant Gaussian curvature surface 
 with $K>-1$ but $K \neq 0$, 
 then $L^{\l}|_{\l = \l_0}$ is a non-conformal harmonic map. 
\end{proof}

\begin{remark}
\mbox{}
\begin{enumerate}
 \item 
  For a flat ($K=0$, that is, $\sigma = 1$) 
 surface in $\h$, the Lagrangian Gauss map
 does not take values in $\mathbb S^2$ nor $\mathbb H^2$, however, the extended frame $\Psi^{\l}$
 is in fact holomorphic with respect to $z$ 
 by a diagonal gauge $D=\diag(e^{u/4}, e^{-u/4})$, that is, 
 $(\Psi^{\l} D)^{-1} \bp (\Psi^{\l} D) = 0$, see \eqref{eq:UVlambda2}.
 Moreover, the Weierstrass type formula is known for flat surfaces in $\h$,  
 \cite{GMM:flat}. Thus from now on, we do not consider flat surfaces.
\item The absolute value $|\lambda_0|$ in Theorem \ref{thm:realHarmonic}
takes values in $(1, \infty)$. Instead one can choose 
\[
 |\lambda_0| = \exp \big(-\operatorname{arcosh} \sqrt{-1/K}\big), 
\quad \mbox{or}\quad
|\lambda_0| = \exp \big(-\operatorname{arsinh} \sqrt{1/K}\big), 
\]
 in case of $-1<K<0$ or in case of $K>0$, respectively.
 Then, $|\lambda_0|$ takes values in $(0, 1)$, and $L^{\lambda}|_{\lambda = \lambda_0}$
 is a harmonic local diffeomorphism into $\mathbb H^2$ or 
 a harmonic map into $\mathbb S^2$, respectively.
\end{enumerate}

\end{remark}

 Conversely let $\hat L$ be a non-conformal harmonic map into $G/K = \mathbb S^2$ or $\mathbb H^2$. There exists a family of maps $\hat \Psi^{\l}$
 taking values in $\LSU$ or $\LISU$ such that
 \eq{
 (\hat \Psi^{\l})^{-1} \mathrm{d} \hat \Psi^{\l}  = \hat \alpha^{\l} = \l^{-1} \hat \alpha_{\mathfrak p}^{\prime}+ \hat \alpha_{\mathfrak k} + \l \hat \alpha_{\mathfrak p}^{\prime \prime}
\quad \mbox{with} \quad \hat \Psi^{\l}|_{z_*} = {\rm id}
}
 holds, that is, $\hat \Psi^{\l}$ is  the extended frame of $\hat L$. 
 Using $\hat \Psi^{\l}$, we can  give a family of constant Gaussian curvature surface 
 in the following theorem.
\begin{theorem}
 Let $\hat L$ be a non-conformal harmonic map into $\mathbb S^2$ or $\mathbb H^2$
 and $\hat \Psi^{\l}$ the corresponding extended frame. Define respective maps 
 \eq{
  \hat f^{\l} =\hat \Psi^{\l} \Vec{e}_0 (\hat \Psi^{\l})^* 
 \quad\mbox{and}\quad
  \hat n^{\l} =\hat \Psi^{\l} \Vec{e}_1 (\hat \Psi^{\l})^*.
}
 The following statements hold$:$
\begin{enumerate}
 \item In case of $\mathbb H^2$$:$
For any $\lambda_1 \in \C^\times$ with $|\lambda_1|> 1$
$($after replacing $\lambda_1$ by $\lambda_1^{-1}$ if necessary, we may assume $|\lambda_1|>1)$, the specialization $\hat f^{\lambda}\big|_{\lambda=\lambda_1}$ is a surface in $\mathbb{H}^3$ with constant Gaussian curvature
 \eq{-1< K= -\left(\frac{ 2|\l_1|}{|\l_1|^2+1}\right)^2< 0.}
 \item In case of $\mathbb S^2$$:$
For any $\lambda_1 \in \C^\times$ with $|\lambda_1|> 1$
$($after replacing $\lambda_1$ by $\lambda_1^{-1}$ if necessary, we may assume $|\lambda_1|>1)$, the specialization 
 $\hat f^{\l}|_{\l = \l_1}$ is a  surface in $\mathbb H^3$ with 
 constant Gaussian curvature 
 \eq{K= \left(\frac{2|\l_1|}{|\l_1|^2-1}\right)^2>0.}

\end{enumerate}
\end{theorem}
\begin{proof}
 The Maurer-Cartan form $\hat \alpha^{\l} = (\hat \Psi^{\l})^{-1} 
 \mathrm{d} \hat \Psi^{\l}$ for the extended frame $\hat \Psi^{\l}$
 of a non-conformal harmonic map $\hat L$ has the following form$:$
 \eq{\hat \alpha^{\l} = \l^{-1} \hat \alpha_{\mathfrak p}^{\prime}+ \hat 
 \alpha_{\mathfrak k} + \l \hat \alpha_{\mathfrak p}^{\prime \prime},}
 where 
\eq{ \alpha_{\mathfrak k}=\frac{1}{4}\left(\p \hat u \> \dz
-\bp \hat u \> \dzb \right)\Vec{e}_1,   \quad 
 \alpha_{\mathfrak p}^{\prime} 
 =  \frac{1}{2}\left(
 e^{\hat u/2} \Vec{\hat e}_2  + \hat Q e^{- \hat u/2}\Vec{\hat e}_3
\right) \dz, \quad 
}
%
 and $\alpha_{\mathfrak p}^{\prime\prime}  
 = \mp {}^t\overline{{\alpha_{\mathfrak p}^{\prime}}}$.
 Here, the plus sign is chosen in case of $\mathbb H^2$ and the 
 minus sign is chosen in case of $\mathbb S^2$.
 Moreover, $\hat Q\; \dz^2$ is a holomorphic quadratic differential of the harmonic 
 map $\hat L$ defined by $\hat Q = \langle \p \hat L, \p \hat L \rangle $
 and the $e^{\hat u(z)}\rho^{-2}(z)$, where $\rho^2(z)= 4/(1-|z|^2)^2$, 
 is the $\p$-energy of $\hat L$.  Note that 
 relation of the coefficient of $\Vec{\hat e}_2$ and the 
 coefficient of $\Vec{e}_1$ is given by the Maurer-Cartan equation $
 (\mathrm{d} \hat \alpha^{\l} + \frac{1}{2}[\hat \alpha^{\l} \wedge \hat \alpha^{\l}])|_{\l = 1}=0$
 and the non-conformality of the harmonic map $\hat L$.

  In case of $\mathbb H^2$: Choose $ |\l_1| >1$, 
 and define $e^{u/2}$ and 
 $Q$ by 
 \eq{
 e^{u/2} = \frac{1+|\l_1|^2}{2 |\l_1|}  e^{\hat u/2}
 \quad \mbox{and}\quad 
 Q = \left(\frac{1+|\l_1|^2}{2 |\l_1|}\right)^2 \hat Q.
}
 It is easy to see that $\p \hat u = \p u$ and $\bp \hat u = \bp u$.
 The Maurer-Cartan form $\hat \alpha^{\lambda}$ at $\lambda_1  
 = |\l_1| e^{i t}, \;(t \in \R)$
\begin{align}\label{eq:alphahat}
\hat \alpha^{\l}|_{\lambda = \lambda_1} = 
 \frac{1}{4} \left(\p u \> \dz  - \bp u \> \dzb\right)\Vec{e}_1
  &+ \frac{1 + \sigma}{2} 
 e^{-i t}\left(e^{u/2} \Vec{\hat e}_2 + Qe^{- u/2} \Vec{\hat e}_3 \right) \dz
  \\ \nonumber &+
 \frac{1 - \sigma}{2} e^{i t} \left(\bar Qe^{- u/2} \Vec{\hat e}_2 + e^{u/2} \Vec{\hat e}_3 \right) \dzb
\end{align}
 with $|\sigma| = |(1-|\l_1|^2)(1+|\l_1|^2)^{-1}|<1$, 
 that is, $\hat \Psi^{\l}|_{\lambda =\lambda_1}$ is the extended frame of some constant Gaussian curvature $-1< K= -1 + \sigma^2 < 0$ surface in $\mathbb H^3$.
 Therefore by using Proposition \ref{prp:associatefamily}, the map 
 $\hat f^{\l}|_{\l_1}$ gives a surface in $\mathbb H^3$ with 
 constant Gaussian curvature \eq{
 -1< K = -1+\sigma^2 = -\left(\frac{ 2 |\l_1|}{|\l_1|^2+1}\right)^2<0.
}
 In case of $\mathbb S^2$: Choose $|\l_1|>1$, 
 and define $e^{u/2}$ and 
 $Q$ by 
 \eq{
e^{u/2}  = \frac{1-|\l_1|^2}{2 |\l_1|}  e^{\hat u/2}
 \quad \mbox{and}\quad 
 Q =  \left(\frac{1-|\l_1|^2}{2 |\l_1|}\right)^2 \hat Q.
}
 It is easy to see that $\p \hat u = \p u$ and $\bp \hat u = \bp u$.
 The Maurer-Cartan form $\hat \alpha^{\lambda}$ at $\lambda_1
 = |\l_1| e^{i t}, \;(t \in \R)$ becomes \eqref{eq:alphahat} 
 with $|\sigma| = |(1+|\l_1|^2)(1-|\l_1|^2)^{-1}|>1$,  
 that is, $\hat \Psi^{\l}|_{\lambda =\lambda_1}$ is the 
 extended frame of some constant Gaussian curvature $K = -1 + \sigma^{-2}>0$ 
 surface in $\mathbb H^3$.
 Therefore by using Proposition \ref{prp:associatefamily}, the map 
 $\hat f^{\l}|_{\l_1}$ gives a surface in $\mathbb H^3$ 
 with constant Gaussian curvature 
\eq{
 K = -1+\sigma^2 = \left(\frac{2 |\l_1|}{|\l_1|^2-1}\right)^2>0.
 }
 This completes the proof.
\end{proof}

\section{Classification of weakly complete constant Gaussian curvature surfaces}\label{sc:Complete}
 It is known that the $\p$-energy density of a harmonic map into $\mathbb H^2$
 is $e^{u} \rho^{-2}$, where $\rho^2 = 4 /(1-|z|^2)^2$ 
 and  the metric $e^{u} \> \dz \dzb$ is a natural metric 
 for the harmonic map, \cite{Wan}.
 Unfortunately, this metric is not directly related to the metric 
 of the constant Gaussian curvature $K>-1$ surface. 
 Thus we consider another natural metric which is induced from sum of the first 
 and third fundamental forms, the so-called \textit{weak metric},
 that is the pullback of the
 Sasaki metric on the unit tangent sphere bundle $\UH$, 
 and we show the equivalence between the 
completeness of the weak metric and that of the metric $e^{u} \dz \dzb$.
 We then give a classification of constant Gaussian curvature $-1 < K <0$ 
 surfaces with the weak metric such that it is complete, 
 the so-called weakly complete  metric. Finally, we show the existence of 
complete and equivariant complete constant Gaussian curvature $-1< K <0$ surfaces.
 \subsection{Weakly complete metrics and a classification of 
 constant Gaussian curvature surfaces}\label{sbsc:classification} 
 Let $f : M \to \h$ be a surface with the unit normal $n: M \to \mathbb S^{1,2}$.
 We define a weakly complete metric of the surface $f$ with $K>-1$.
\begin{definition}
 For a surface $f$ with $K>-1$, a metric 
\begin{equation}\label{eq:weakmetric}
 \mathrm{d} t^2 = 
 \langle \mathrm{d} f, \mathrm{d} f\rangle +\frac1{1+K} \langle \mathrm{d} n, \mathrm{d} n\rangle
\end{equation} 
 is called the \textit{weak metric} and 
 if the weak metric $\mathrm{d} t^2$ is complete, it is 
 called the \textit{weakly complete metric} and the surface
 $f$ is called the \textit{weakly complete surface}.
\end{definition}
\begin{remark}
\mbox{}
\begin{enumerate}
\item The notion of weak completeness was originally defined for a 
 constant negative Gaussian curvature surfaces in $\R^3$, see \cite{MS}.

\item For $K=0$, the weak metric is a pullback of the 
 \textit{Sasaki metric} of the unit tangent sphere bundle $\UH$, 
 $\mathrm{d} t^2 = \langle \mathrm{d} f, \mathrm{d} f\rangle + \langle \mathrm{d} n, \mathrm{d} n\rangle$.
 Note that the natural metric of the Legendrian immersion 
 $F= (f, n)$ is given by $\langle \mathrm{d} F, \mathrm{d} F \rangle= \langle \mathrm{d} f, \mathrm{d} f\rangle - \langle \mathrm{d} n, \mathrm{d} n\rangle$, which is indefinite,
\cite{DIK1}.
\end{enumerate}
\end{remark}
\begin{lemma}\label{lm:completeness}
 Let $f: M \to \h$ be a constant Gaussian curvature $K>-1$ surface.
 Moreover let $\mathrm{d}s^2 = e^u \> \dz \dzb$ be the metric defined by 
 the function $e^u$ given in \eqref{eq:firstsecond2}.
 The surface $f$ is a weakly complete surface if and only 
 if the metric $\mathrm{d}s^2$ is complete.
\end{lemma}
\begin{proof}
 Let $\mathrm{d} t^2 = \mathrm{I} + \frac1{1+K}\mathrm{I\!I\!I}$ be the weak metric of the 
 surface $f$, where $\mathrm{I} = \langle \mathrm{d} f, \mathrm{d} f\rangle $ 
 and $\mathrm{I\!I\!I} = \langle \mathrm{d} n, \mathrm{d} n\rangle$ are the first and third fundamental forms, respectively. 
 It is then easy to see that,  by \eqref{eq:firstsecond2} and \eqref{eq:third2},  $\mathrm{d} t^2$ can be computed as
\eq{
\mathrm{d} t^2 = 2 (e^{u} + |Q|^2 e^{-u}) \> \dz \dzb.
}
 Moreover, since $f$ is a constant Gaussian curvature, if necessary by choosing 
 the unit normal as $-n$, we have $e^{2u} > |Q|^2$, see \eqref{eq:firstsecond2}.
 If $e^u \> \dz \dzb$ is complete, then $\mathrm{d} t^2$ is complete, 
that is, the surface is a weakly complete surface. 
 Conversely if $\mathrm{d} t^2$ is complete, then 
 $\displaystyle \int_C \sqrt{2 (e^{u} + |Q|^2 e^{-u})} |\dz|$
  needs to diverge along any divergent
 curve $C$ on $M$, that is, either 
 \eq{\sqrt{2}\int_C e^{u/2} |\dz|\quad \mbox{or} \quad \sqrt{2}\int_C e^{-u/2} |Q| |\dz|} 
 needs to diverge.
 The latter case also implies that $\displaystyle \int_C e^{u/2}|\dz|$ diverges
 for any divergent curve $C$ on $M$, 
 since $ e^{u}> |Q|$. This completes the proof.
\end{proof}
 As stated Theorem \ref{thm:classification} in Introduction, we 
 classify weakly complete 
 constant negative Gaussian curvature $-1<K<0$ surfaces.
%
\begin{proof}[Proof of {\rm Theorem \ref{thm:classification}}]
 From Theorem \ref{thm:Ruh-Vilms}, for a constant negative Gaussian curvature surface 
 $-1< K<0$ in $\h$, the Klotz differential is holomorphic. Since the equivalent 
 class of the surface, that is a rigid motion, does not change the Klotz differential, thus we can define a map 
 $g$
 from $\mathcal K$ into $\mathcal {QD}$.
 
 Let us consider the Gauss-Codazzi equations 
\begin{equation}\label{eq:GCK}
 \bp \p u+ \frac{K}{2}( e^u -|Q|^2e^{-u})=0,\ \ \bp Q=0
\end{equation}
 for a constant Gaussian curvature $-1< K<0$ surface $f$.
 Note that the second equation is nothing but holomorphicity equation of $Q$. 
 It is well known that, using a conformal change of coordinates $z \to \sqrt{-K} z$, 
 the first equation is the elliptic PDE for 
 harmonic maps from $\D$ or $\C$ into $\mathbb H^2$, see \cite{Wan, 
 WanAu}.
 It is also well known that for a given quadratic differential $Q \> \dz^2$ 
 on $\D$ or $\C$ which is not identically zero on $\C$, 
 there exists a unique solution $u$ of the first equation in \eqref{eq:GCK} 
 such that 
 $e^u \> \dz \dzb$ is complete. 
 Moreover, it has been shown in \cite[Theorem 3.2]{TW} that 
 if the holomorphic quadratic differential $Q \> \dz^2$ 
 is not constant on $\mathbb C$, then 
 the corresponding harmonic map is global diffeomorphism and 
 satisfies the inequality 
 \eq{e^{2u} >|Q|^2,}
 see \cite[Lemma 1.3 (2)]{TW}.
 Therefore by fundamental theorem of surface theory 
there exists a unique constant Gaussian curvature $-1< K
 = -1 + \sigma^2 <0$
 surface $f$, up to rigid motion, such that the holomorphic 
 Klotz differential is $Q \> \dz^2$
 and ${\rm I}^{(1, 1)} + \sigma^{-1} {\rm I\!I}^{(1, 1)} =2 e^u \> \dz \dzb$,
 where the superscript $(1, 1)$ denotes the $(1, 1)$-part of the fundamental forms.
 Note that since $\det \tilde {\rm I} = (e^u -|Q|^2e^{-u})^2 >0$,
 where $\tilde {\rm I}$ is the coefficient matrix of the first fundamental 
 form ${\rm I}$, $f$ never degenerate.
 Moreover, from Lemma \ref{lm:completeness}, the resulting surface $f$ 
 is a weakly complete surface.  
 Note that since the equivalence class 
 of the holomorphic  quadratic differential 
 is given by a M\"obius transformation on $\D$ or $\C$,
 it corresponds to change of conformal parameters of the corresponding constant 
 Gaussian curvature surface.
 Therefore,  $g: \mathcal K \to \mathcal {QD}$ is bijective.
 This completes the proof.
\end{proof}
\begin{remark}
\mbox{}
\begin{enumerate}
 \item 
  In general, weakly complete constant Gaussian curvature $-1< K<0$ surfaces
 would have non-trivial topology. The most simplest example is an equidistant cylinder 
 which have the non-zero constant Klotz differential. We discuss 
 further examples in a forthcoming paper. 

\item  If the holomorphic quadratic differential is non-zero constant 
 on $\mathbb C$, then the corresponding harmonic map is not diffeomorphism, 
 and in fact $e^{2u} = |Q|^2= c \neq 0$ holds and the corresponding surface 
 becomes degenerate, that is, a geodesic in $\h$. 

\item Theorem \ref{thm:classification} does not hold for weakly complete constant 
 Gaussian curvature $K>0$ surfaces in $\h$, since there are 
 infinitely many such surfaces for a given holomorphic quadratic differential $Q \>\dz^2$.
 Moreover, the resulting surface could be degenerate at some point, 
 that is, $e^{2u }= |Q|^2$ occurs.
 \end{enumerate}

\end{remark}
\subsection{Complete constant Gaussian curvature surfaces}
  As corollaries of Theorem \ref{thm:classification} we show the 
  existence of complete and  equivariant constant Gaussian 
 curvature $-1< K <0$ surfaces in $\h$, respectively.
\begin{corollary}
 Let $Q \> \dz^2$ be a bounded holomorphic quadratic differential on the unit disk 
 $\D$ with 
 respect to the Poincare metric $\mathrm{d} s_p^2 = \rho^2(z) \dz \dzb$ with 
 $\rho^2(z) = 4(1-|z|^2)^{-2}$, that is, 
 \eq{
 \sup_{z \in \D}|Q(z)|\rho^{-2}(z) < \infty.  
 }
 There exists a unique complete constant Gaussian curvature $-1< K<0$ surface 
 in $\h$ whose Klotz differential is $Q \> \dz^2$.
\end{corollary}
\begin{proof}
 We need to prove completeness of the constant Gaussian curvature surface $f$
 with the bounded Klotz differential $Q\> \dz^2$.
 The first fundamental form ${\rm I}$ of the surface $f$ constructed in the proof 
 of Theorem \ref{thm:classification} is given as 
\eq{\mathrm{I}=Q \> \dz^2+(e^u +|Q|^2e^{-u})\> \dz \dzb
+\bar Q\>\dzb^2.
}
 It can be  rephrased as 
 \eq{\mathrm{I}=\frac1{|K|}e^{u} \left|
 \mathrm{d}{\hat z}+ \mu \mathrm{d} \bar{\hat z}
 \right|^2, \quad \mu= \frac{\bar Q}{e^u}, \quad \hat z = \sqrt{-K} z. 
}
 Note that the coordinates $\hat z$ normalizes the Gauss equation in \eqref{eq:GCforCGC}
 as the standard harmonic map setting, and for simplicity, we write $\hat z$ by $z$. 
 Since $Q \> \dz^2$ is bounded,  by \cite[Theorem 2.1]{TW}, 
 there exists a quasi-conformal diffeomorphism $\zeta$ on $\D$ 
 which satisfies the Beltrami equation $\bp \zeta = \mu \p \zeta$.
 Here $\mu$ can be rephrased as
\[
 \mu = \frac{\bar Q}{e^u } = \frac{\bar Q \rho^{-2}}{e^u \rho^{-2}},
\quad \rho^2(z) = \frac4{(1-|z|^2)^{2}},
\]
 and $e^u \rho^{-2}$ is the $\p$-energy of the deformed Lagrangian Gauss map 
 $L^{\lambda}|_{\lambda = \lambda_0}$ into $\mathbb H^2$. Moreover by \cite[Lemma 2.4]{TW} 
 the solution $\zeta$ is a harmonic diffeomorphism onto $(\D, \rho^2|d\zeta|^2)$ 
 and the $\p$-energy of $\zeta$ is again $e^u \rho^{-2}$.
 Therefore the first fundamental $\mathrm{I}$ can be rephrased in the 
 coordinates $\zeta$ as
\begin{equation}\label{eq:Iwithw}
 \mathrm{I} =
\frac1{|K|}e^{u} \left|
 \mathrm{d}{z}+ \mu \mathrm{d} \bar{z}
 \right|^2 =  \frac1{|K|}\frac{e^u}{|\p \zeta|^{2}} |\mathrm{d} \zeta|^2
= \frac1{|K|}\frac{e^u \rho^{-2}}{|\p \zeta|^{2}} \rho^2 |\mathrm{d} \zeta|^2
=  \frac1{|K|}\rho^2 |\mathrm{d} \zeta|^2.
\end{equation}
 Thus $\mathrm{I}$ defines a complete metric.
 This completes the proof.
\end{proof}
%
 From Gauss-Bonnet theorem, it is easy to see that there does not exist 
 a constant Gaussian curvature $-1< K< 0$ surface in $\h$ of genus zero.
 Therefore in the following corollary, we omit the genus zero case.
\begin{corollary}\label{coro:equivariance}
 Let $Q \> \dz^2$ be a holomorphic quadratic differential defined on a Riemann surface
 $M = \widetilde M/\Gamma$, where $\widetilde M$ is the unit disk $\D$ or $\C$ and 
 $\Gamma$ is a Fuchsian group, respectively. Moreover assume that in case of 
 $\widetilde M = \C$, $Q \> \dz^2$ is not constant.
 There exists a unique 
 weakly complete 
 constant Gaussian curvature $-1<K<0$  surface $f: \widetilde M \to \h$
 satisfying the following equivariant property$:$
 There exist a representation $\rho$ from $\Gamma$
 to the isometry group $\operatorname{Iso} \h$
 such that for any $\gamma \in \Gamma$ the following relation holds$:$
 \eq{
f\circ \gamma  =  \rho(\gamma) \circ f.
}
\end{corollary}
\begin{proof}
 From Theorem \ref{thm:classification}, for a given holomorphic quadratic differential 
 $Q \> \dz^2$ on $\D$ or $\C$, there exists a unique 
 weakly complete constant Gaussian curvature $-1<K = -1 + \sigma^2 <0$ 
 surface $f$ from $\D$ or $\C$ into $\h$ such that the Klotz differential 
 of $f$ is $Q \> \dz^2$. 
 Moreover since $Q \> \dz^2$ is invariant with respect to 
 $\Gamma$, that is,  for any $\gamma \in \Gamma$,  $Q \circ \gamma = Q$ holds and 
 the uniqueness of the surface $f$, there exists
 a representation $\rho: \Gamma \to \operatorname{Iso} \h$ such 
 that for any $\gamma \in \Gamma$  
 \eq{
  f \circ \gamma = \rho(\gamma) \circ f
}
 holds. This completes the proof.
\end{proof}
\begin{remark}
\mbox{}
\begin{enumerate}
 \item 
 In \cite[Theorem D]{La}, a similar statement to Corollary \ref{coro:equivariance} 
 has been proved using completely different idea, the so-called 
``asymptotic Plateau problem'', that is, the existence of 
 a constant Gaussian curvature $-1< K<0$ surface 
 which is asymptotically to a given curve $\gamma$ 
 in the ideal boundary of $\mathbb H^3$.
\item
 To show the existence of 
 weakly complete constant Gaussian curvature $-1< K<0$ 
 surfaces $f$ which have nontrivial topology, 
 we need 
 \eq{
 \rho (\gamma) = \operatorname{id}
}
 for some $\gamma \in \Gamma$. It is not easy to see this condition holds for a
 given holomorphic quadratic differential.
 \end{enumerate}
 \end{remark}

\appendix
\section{Loop groups}\label{ap:loopgroups}
 Let $\SL$ be the special linear group of degree $2$ and its Lie algebra 
 $\slt$. The \textit{twisted loop group} of $\SL$ is a set of 
 smooth maps 
\eq{
 \LSL =  \{ g : \mathbb S^1 \to \SL \;|\; \tau g(\lambda) = g(-\lambda)  \},
}
 where $\tau$ is defined by $\tau g = \mathrm{Ad} (\Vec{e}_1) g$.
 It is known that $\LSL$ is a complex Banach Lie group with respect to 
 a suitable topology, see \cite{LoopGroup}. The Lie algebra of $\LSL$ is 
 \eq{
 \lslt =  \{ g : \mathbb S^1 \to \slt \;|\; \tau g(\lambda) = g(-\lambda)  \},
}
 and it is called the \textit{twisted loop algebra} of $\LSL$.
 We then define two subgroups $\LSU$ and $\LISU$ as 
\eq{
\LSU &= \{ g\in \LSL\;|\;   {}^t\overline {g(1/\bar \l)}^{-1} = g(\l) \}, \\
\LISU &= \{g\in \LSL\;|\; \Vec{e}_1 {}^t \overline{g(1/\bar \l)}^{-1}\Vec{e}_1 
 = g(\l)\},
}
 and they are called the \textit{twisted loop groups} of $\SU$ and $\ISU$, 
 respectively. It is easy to see that these subgroups are real forms 
 of the complex Banach Lie group $\LSL$, that is, they are given by fixed point 
 sets of anti-linear involutions on $\LSL$.
\section{Harmonic maps into symmetric spaces}\label{ap:harmonic}
 It is known that harmonic maps from a Riemann surface into 
 a symmetric space can be characterized by loop groups, see,  e.g., \cite{BP}.
 In this paper we consider two particular symmetric spaces, that is, 
 the $2$-sphere $\mathbb S^2$ and the hyperbolic $2$-space $\mathbb H^2$.
 Thus we briefly explain how harmonic maps into these symmetric spaces
 are characterized.

 Let $M$ be a Riemann surface and 
 $G/K$ be a symmetric space $\mathbb S^2 = \SU/\Uone$ or 
 $\mathbb H^2 = \ISU/\Uone$. Moreover, let $\varphi : M \to G/K$
 be a smooth map and $F$ a frame taking values in $G$. Moreover, let 
 $\alpha = F^{-1} \mathrm{d} F$ be the Maurer-Cartan form of $F$.
 It is known that $\mathrm{d} + \alpha$ gives a flat connection, that is, 
 the partial differential equation
\eq{
 \mathrm{d} \alpha + \frac{1}{2}[\alpha \wedge \alpha]=0
}
 is satisfied.
 We decompose 
 the Maurer-Cartan form $\alpha = F^{-1} \mathrm{d} F$ by 
 \eq{
 \alpha = \alpha_{\mathfrak k} + \alpha_{\mathfrak p}^{\prime} +  
 \alpha_{\mathfrak p}^{\prime \prime},
}
 according to the decomposition $\mathfrak g = \mathfrak k \oplus \mathfrak p$
 of the Lie algebra $\mathfrak g$ and $\mathfrak k$ of $G$ and $K$, respectively,
 and $\prime$ and $\prime \prime$ denotes 
 $(1, 0)$- and $(0, 1)$-parts on $M$, respectively. Define 
 a family of Maurer-Cartan forms $\alpha^{\l}$ as
\eq{
\alpha^{\l} =\l^{-1}\alpha_{\mathfrak p}^{\prime} + \alpha_{\mathfrak k} +  
 \l  \alpha_{\mathfrak p}^{\prime \prime}, \quad \l \in \C^{\times}.
}
 From the construction, it is clear to see that $\mathrm{d} + \alpha^{\l}|_{\l =1}$
 gives the original flat connection $d + \alpha$. The harmonicity of 
 the map $\varphi$ is characterized as follows.
\begin{theorem}[\cite{BP}]\label{thm:flatconnections}
 A smooth map $\varphi : M \to G/K$ is harmonic if and only if 
 $\mathrm{d} + \alpha^{\l}$ is a flat connection, that is, the partial 
 differential equation 
 \eq{
 \mathrm{d} \alpha^{\l} + \frac{1}{2}[\alpha^{\l} \wedge \alpha^{\l}] =0
 }
 is satisfied.
\end{theorem}
  
  \section{Harmonic maps into symmetric spaces}\label{ap:harmonic}
\subsection{}
We recall that harmonicity from a Riemann surface $M$ 
 into a homogeneous semi-Riemannian space $N = G/H$. 
 For more detail, see, e.g., \cite{BP,DIK1}. 
 Let $\varphi:M\to G/H$ be a smooth map of a simply connected Riemann surface 
 into a reductive homogeneous semi-Riemannian space $N = G/H$ equipped 
 with a reductive decomposition $\mathfrak{g}
 =\mathfrak{h}\oplus\mathfrak{p}$. 
 There exists a smooth map $\Phi:M\to G$ such that 
 $\Phi H=\varphi$. Such a map $\Phi$ is called a 
 \emph{frame} of $\varphi$. The Lie algebra 
 $\mathfrak{g}$-valued $1$-form 
 $\alpha=\Phi^{-1}\mathrm{d}\Phi$ is called the 
 \emph{Maurer-Cartan form} of $\Phi$. 
 The Maurer-Cartan form $\alpha$ defines a 
 flat connection $\mathrm{d}+\alpha$ on the 
 principal bundle $M\times G$. The flatness is 
 expressed as the following \textit{zero curvature equation}:
 
\[
\mathrm{d} \alpha + \frac{1}{2}[\alpha \wedge \alpha]=0.
\]
 We decompose 
 the Maurer-Cartan form $\alpha$ by 
 \eq{
 \alpha = \alpha_{\mathfrak k} + \alpha_{\mathfrak p},
 \quad 
 \alpha_{\mathfrak p}=
 \alpha_{\mathfrak p} ^{\prime} +  
 \alpha_{\mathfrak p}^{\prime \prime},
}
 according to the reductive 
 decomposition $\mathfrak{g} = \mathfrak{k}\oplus\mathfrak{p}$
 and $\prime$ and $\prime \prime$ denotes 
 $(1, 0)$- and $(0, 1)$-parts on $M$, respectively.

 The harmonicity of a map $\varphi$ can be characterized by the Maurer-Cartan form 
 $\alpha$ as 
\begin{equation}\label{eq:harmonicity}
[\alpha^{\prime}_{\mathfrak p}\wedge 
\alpha^{\prime\prime}_{\mathfrak p}]_{\mathfrak p}=0= 
\mathrm{d}(*\alpha_\mathfrak{p})
+[\alpha\wedge *\alpha_{\mathfrak p}],
\end{equation}
 where the subscript $\mathfrak h$ and $\mathfrak p$ 
 denote the $\mathfrak h$- and $\mathfrak p$-part, respectively. 
 Moreover $*$ denotes the Hodge star operator, for example $* \dz = -i \dz$ and 
 $* \dzb = i \dzb$.
 In particular if the target homogeneous space $G/H$ is a semi-Riemannian symmetric space, then the 
 first condition is vacuous.
\subsection{}
 It is known that harmonic maps from a Riemann surface into 
 a semi-Riemannian symmetric space can be characterized by loop groups, see, e.g., 
 \cite{BP}.
 In this paper we consider two particular symmetric spaces, that is, 
 the $2$-sphere $\mathbb S^2$ and the hyperbolic $2$-space $\mathbb H^2$.
 Thus we briefly explain how harmonic maps into these symmetric spaces
 are characterized.

 Let $M$ be a simply connected Riemann surface and 
 $G/K$ be a Riemannian symmetric space $\mathbb S^2 = \SU/\Uone$ or 
 $\mathbb H^2 = \ISU/\Uone$. Moreover, let $\varphi : M \to G/K$
 be a smooth map and $\Phi$ a framing of $\varphi$. 
 Consider the Maurer-Cartan form 
 $\alpha=\Phi^{-1}\mathrm{d}\Phi=
 \alpha_{\mathfrak k} + \alpha_{\mathfrak p}^{\prime}+
 \alpha_{\mathfrak p}^{\prime\prime}$ and define 
 a family of Maurer-Cartan forms $\alpha^{\l}$ as
\[
\alpha^{\l} =\l^{-1}\alpha_{\mathfrak p}^{\prime} + \alpha_{\mathfrak k} +  
 \l  \alpha_{\mathfrak p}^{\prime \prime}, \quad \l \in \C^{\times}.
\]
 From the construction, it is clear to see that $\mathrm{d} + \alpha^{\l}|_{\l =1}$
 gives the original flat connection $d + \alpha$. The harmonicity of 
 the map $\varphi$ is characterized as follows.
\begin{theorem}[\cite{BP}]\label{thm:flatconnections}
A smooth map $\varphi : M \to G/K$ of a simply connected Riemann surface into a semi-Riemannian symmetric space $G/K$ is harmonic if and only if 
 $\mathrm{d} + \alpha^{\l}$ is a flat connection for all $\lambda$, that is, the partial 
 differential equation 
 \eq{
 \mathrm{d} \alpha^{\l} + \frac{1}{2}[\alpha^{\l} \wedge \alpha^{\l}] =0
 }
 is satisfied for all $\lambda \in \C^{\times}$.
\end{theorem}
  

\textbf{Conflict of interest:} On behalf of all authors, the corresponding author states that there is no conflict of
interest. No new data were created or analyzed in this study.

\bibliographystyle{plain}

\end{document}